\documentclass[10pt]{amsart}
\usepackage{amsmath, amscd, amsthm} 
\usepackage{amssymb, amsfonts} 

\usepackage[all]{xy} 
\usepackage{color}
\newcommand{\Q}{\mathbb{Q}} 
\newcommand{\C}{\mathbb{C}} 
\newcommand{\A}{\mathbb{A}}
\renewcommand{\P}{\mathbb{P}}
\newcommand{\Z}{\mathbb{Z}}
\newcommand{\R}{\mathbb{R}}
\newcommand{\G}{\mathbb{G}} 

\renewcommand{\L}{\mathbb{L}}
\newcommand{\N}{\mathbb{N}}

\renewcommand{\SS}{\mathbb{S}}

\newcommand{\QQ}{\mcal{Q}}

\newcommand{\II}{\mcal{I}}
\newcommand{\LL}{\mcal{L}}

\newcommand{\Qsst}{\mcal{Q}^{\mathrm{sst}}}

\newcommand{\nn}{\mathbf{n}}
\newcommand{\mm}{\mathbf{m}}

\renewcommand{\Re}{\mathrm{Re}}
\newcommand{\RRe}{\mathrm{Re}}

\newcommand{\iso}{\cong}
\newcommand{\wkeq}{\simeq}

\newcommand{\sst}{\mathrm{sst}}
\newcommand{\Sch}{\mathrm{Sch}}
\newcommand{\Sm}{\mathrm{Sm}}
\newcommand{\sch}{\mathrm{Sch}}
\newcommand{\sm}{\mathrm{Sm}}

\newcommand{\sSet}{\mathrm{sSet}}
\newcommand{\sset}{\mathrm{sSet}}
\newcommand{\ssetb}{\mathrm{sSet}_{\bullet}}
\newcommand{\btop}{\mathrm{Top}_{\bullet}}
\newcommand{\Top}{\text{Top}}
\newcommand{\ab}{\mathrm{Ab}}

\newcommand{\Th}{\mathrm{Th}}
\newcommand{\id}{\mathrm{id}}

\newcommand{\MGL}{\mathrm{MGL}}
\newcommand{\KGL}{\mathrm{KGL}}
\newcommand{\rKGL}{\mathcal{K}\mathrm{GL}}
\newcommand{\kgl}{\mathrm{kgl}}
\newcommand{\BGL}{\mathrm{BGL}}
\newcommand{\BGmp}{\mathrm{B}\mathbb{G}_{m\,+}}
\newcommand{\BGm}{\mathrm{B}\mathbb{G}_{m}}
\newcommand{\MZ}{\mathbf{M}\mathbb{Z}}
\newcommand{\ku}{ku}
\newcommand{\KU}{KU}
\newcommand{\Grass}{\mathrm{Grass}}

\newcommand{\free}{\mathrm{Free}^{\II}}

\newcommand{\mcal}[1]{\mathcal{#1}}

\DeclareMathOperator*{\colim}{\mathrm{colim}}

\DeclareMathOperator*{\hocolim}{\mathrm{hocolim}}

\DeclareMathOperator{\hofib}{\mathrm{hofib}}
\DeclareMathOperator{\cofiber}{\mathrm{cofib}}

\DeclareMathOperator{\diag}{\mathrm{diag}}

\DeclareMathOperator{\spec}{\mathrm{Spec}}

\DeclareMathOperator{\Mor}{Mor}

\DeclareMathOperator{\sk}{sk}
\DeclareMathOperator{\sing}{\mathrm{Sing}}
\DeclareMathOperator{\msing}{\mathcal{S}\mathrm{ing}}

\newcommand{\Lhk}{\mathrm{L}_{h\mathrm{K}}}

\newcommand{\spt}{\mathrm{Spt}}
\newcommand{\Spt}{\mathrm{Spt}}
\newcommand{\sspt}{\mathrm{Spt}^{\Sigma}}
\newcommand{\MMb}{\mathrm{Spc}_{\bullet}}
\newcommand{\hb}{\mathrm{H}_{\bullet}}

\newcommand{\SH}{\mathrm{SH}}
\newcommand{\DM}{\mathrm{DM}}

\newcommand{\ppsh}{\underline{\pi}}

\newcommand{\shift}{\mathrm{sh}}
\newcommand{\cof}{\mathrm{cof}}
\newcommand{\fib}{\mathrm{fib}}
\newcommand{\pt}{\mathrm{pt}}

\newcommand{\shom}{\underline{\mathrm{Hom}}}
\newcommand{\ihom}{\mathbf{hom}}

\numberwithin{equation}{section} 

\theoremstyle{plain}
\newtheorem{theorem}[equation]{Theorem}
\newtheorem*{theorem*}{Theorem}
\newtheorem{proposition}[equation]{Proposition}
\newtheorem{lemma}[equation]{Lemma}
\newtheorem{corollary}[equation]{Corollary}

\theoremstyle{definition}

\newtheorem{definition}[equation]{Definition}

\newtheorem{remark}[equation]{Remark}

\begin{document}
\title[Semi-topological strict ring spectra]{Motivic strict ring spectra representing semi-topological cohomology theories}
\author{Jeremiah Heller}
\email{heller@math.uni-wuppertal.de}
\address{Bergische Universit\"at Wuppertal, Gau{\ss}str. 20, D-42119 Wuppertal, Germany}
\thanks{The author received partial support from DFG grant HE6740/1-1}
\begin{abstract}
We show that Shipley's ``detection functor'' for symmetric spectra generalizes to motivic symmetric spectra. As an application, we construct motivic strict ring spectra  representing morphic cohomology, semi-topological $K$-theory, and semi-topological cobordism for complex varieties. As a further application to semi-topological cobordism, we show that it is related to semi-topological $K$-theory via a Conner-Floyd type isomorphism and that after inverting a lift of the Friedlander-Mazur $s$-element in morphic cohomology, semi-topological cobordism becomes isomorphic to periodic complex cobordism. 
\end{abstract}

\maketitle
\tableofcontents

\section{Introduction}
Motivic homotopy theory has been a very successful generalization of classical homotopy theory into the algebro-geometric setting and has shown itself to be the appropriate setting in which to  analyze, or even define, many interesting algebro-geometric invariants.
An understanding of a cohomology theory is intimately linked to an understanding of the object which represents it. Working in the modern setting of highly structured ring spectra yields representing objects which capture and reflect fine algebraic structure of the cohomology theory of interest. This has important calculational and theoretical ramifications and has been the source of much exciting development in homotopy theory in the past two decades.

The goal of this paper is twofold. In the second part of the paper, we study certain cohomology theories for complex varieties in the motivic stable homotopy category. In the course of their construction,  it is necessary to have a lax symmetric monoidal fibrant replacement functor if one is to obtain motivic strict ring spectra as representing objects. (The terminology ``strict ring spectra'' is used here to emphasize that we consider monoids in the category of motivic symmetric spectra and not merely after passage to the homotopy category). The usual fibrant replacement functors for motivic symmetric spectra are produced via the small object argument and so aren't suitable for this purpose. Thus the first part of the paper is devoted to importing a  construction of Shipley into the motivic setting, which then leads to  a lax symmetric monoidal fibrant replacement functor for motivic symmetric spectra. This is carried out in Section \ref{sec:fibrant}.

The two main examples of a semi-topological cohomology theory are morphic cohomology, introduced by 
Friedlander-Lawson in \cite{FL:algco} and  semi-topological $K$-theory introduced by Friedlander-Walker \cite{FW:sstfct}. In \cite{thesis} we introduced a semi-topological cobordism theory.  These cohomology theories form a rather interesting class of invariants for complex algebraic varieties, which are linked to hard and important problems in algebraic geometry. For example, Beilinson \cite{Beilinson} shows a certain conjecture of Suslin regarding the comparison between morphic cohomology and singular cohomology (analogous to the Beilinson-Lichtenbaum conjectures relating motivic cohomology and \'etale cohomology) implies Grothendieck's standard conjectures.

Friedlander-Walker originally define semi-topological $K$-theory of a normal complex variety $X$ in terms of the homotopy groups of the homotopy group completion of $\Mor(X,\Grass)$, the set of algebraic maps equipped with a natural topology. When $X$ is projective, this is the analytic topology associated to the set of complex points of a certain ind-scheme. This definition is intuitively appealing but hard to work with in practice. In order to further analyze these theories, Friedlander-Walker \cite{FW:compK} introduce a construction which they call the ``singular semi-topological complex'' for a presheaf $F$ on the category $\Sch/\C$ of schemes over $\C$. This construction is a model for topological realization and the value of $F^{\sst}(X)$ is the topological realization of the functor mapping space $\ihom(X,F)$. From this viewpoint, it is not very hard to see how to define representing spectra for the semi-topological cohomology theories. One simply takes any fibrant model for $E$ and defines $\QQ^{\sst}E$ to be the result of Friedlander-Walker's construction applied levelwise. The fibrant replacement is necessary as $\ihom(X,F)$  only has homotopical meaning when $F$ is fibrant. using the explicit fibrant replacement functor constructed in Section \ref{sec:fibrant}, this construction preserves strict ring spectra.  

In Section \ref{sec:sst} we formally define the motivic version $\Qsst$ of Friedlander-Walker's construction and record some first properties. Two key features of morphic cohomology and semi-topological $K$-theory are that they factor the topological realization map and that with finite coefficients they agree with the algebraic theory. These are completely general properties of $\Qsst$, as we verify in Theorem \ref{thm:sstdef}. Topological realization 
$\L\RRe_{\C}:\SH(\C)\to \SH$ has a right adjoint $\msing_{\C}$ and the unit of the adjunction factors as $\id \to \Qsst \to \R\msing_{\C}\L\RRe_{\C}$. This is rather formal. The second property that $\id\to \Qsst\wedge MA$ is an equivalence, where $MA$ is a Moore spectrum for a finite abelian group $A$, is a consequence of a Suslin rigidity theorem for motivic spectra.

To relate the construction $\Qsst$, defined on the level of motivic spectra, to Friedlander-Walker's construction, defined on the level of motivic spaces, it is necessary to recognize when $\Qsst$ produces an $\Omega_{\P^{1}}$-spectrum. This is carried out in Corollary \ref{cor:omega}. If
a motivic spectrum $E$ satisfies a certain connectivity hypothesis, namely that 
$[\Sigma^{i,-q}X_{+},E]_{\SH(\C)} = 0$ for $i<-2q$ and $q\geq 0$, then $\Qsst E$ is an $\Omega_{\P^{1}}$-spectrum. It is then an immediate consequence of Friedlander-Walker's work that $\Qsst\MZ$ represents morphic cohomology and $\Qsst\KGL$ represents semi-topological $K$-theory, where $\MZ$ (resp. $\KGL$) are motivic spectra representing motivic cohomology (resp. algebraic $K$-theory). Since the motivic spectra $\MZ$, $\KGL$, and $\MGL$ all have models which are commutative motivic strict ring spectra we immediately obtain commutative motivic strict ring spectra representing morphic cohomology, semi-topological $K$-theory, and semi-topological cobordism.

A further fundamental property of morphic cohomology and semi-topological $K$-theory is that 
$L^{q}H^{2q}(X)$ and $K_{0}^{sst}(X)$  are computable respectively in terms of the group of codimension $q$-cycles modulo algebraic equivalence and the Grothendieck group of vector bundles modulo algebraic equivalence. In Proposition \ref{prop:algeq} we show that if $E^{k,q}(X) = 0$ for all smooth $X$ and $k>2q$,  then $E^{2q,q}_{\sst}(X) = E^{2q,q}(X)/\sim_{alg}$.
This includes the case of cobordism and so we have that for smooth $X$
$$
\Omega^{q}(X)/\sim_{alg} \xrightarrow{\iso}\MGL^{2q,q}(X)/\sim_{alg} \xrightarrow{\iso}  \MGL_{\sst}^{2q,q}(X).  
$$

The final fundamental property of morphic cohomology and semi-topological $K$-theory is that $L^{q}H^{*}(\C)=H^{*}(\pt)$ (for $q\geq 0$) and that $K^{\sst}_*(\pt)=ku^{-*}(\pt)$. 
In Theorem \ref{thm:toppt} we generalize this to show that if $E$ is in $\SH(\C)^{eff}$, then
$$
E_{\sst}^{-p,q}(\C)= \ppsh_{p,-q}(\Qsst E)(\C) \xrightarrow{\iso} \pi_{p}\L\Re_{\C}E, 
$$
for $q\geq 0$. Note that this theorem does not apply to $\KGL$ but it does apply to the $\P^{1}$-connective $K$-theory $\kgl:=f_{0}\KGL$. This result can be viewed as an integral extension of 
a result of Levine \cite[Theorem 7.1]{Levine:comp} and as in loc.~cit., we rely on  the convergence of Voevodsky's slice tower over $\C$, proved by Levine (more generally for fields with finite cohomological dimension). Theorem \ref{thm:toppt} also applies to cobordism, however this theory has plenty of interesting information in weights $q\leq 0$. A full computation of $\MGL_{\sst}^{p,q}(\C)$ is given in Remark \ref{rem:fullpt},
$$
\MGL_{\sst}^{p,b}(\C) = \begin{cases}
                         MU^{p} & p\leq 2b \\
0 & \textrm{else}.
                        \end{cases}
$$
This is computed using a comparison of the spectral sequence arising from an application of $\Qsst$ to the slice spectral sequence of $\MGL$ and the Atiyah-Hirzebruch spectral sequence for $MU$.

Two further applications to semi-topological cobordism are given in Section \ref{sec:bott}, relying on the product structure on $\Qsst\MGL$. The Friedlander-Mazur $s$-element in $L^{1}H^{0}(\C)$ lifts to an element $s\in \MGL_{\sst}^{0,1}(\C)$ and upon inverting this element we have
$$
(\oplus_{q}\MGL_{\sst}^{2q+*,q}(X))[s^{-1}] = \oplus_{q}MU^{2q+*}(X).
$$
The second is a semi-topological Conner-Floyd isomorphism 
$$
\MGL_{\sst}^{*}(X)\otimes_{\MGL_{\sst}^{0}}K^{\sst}_{0}= K_{-*}^{\sst}(X),
$$
The existence of these semi-topological cohomology theories suggests that one might hope for  a ``semi-topological'' homotopy theory, in which Friedlander-Walker's construction could be viewed as a fibrant replacement functor. Whether this is possible or not is an entirely open problem. For example, it is not even known whether $\Qsst\Qsst\KGL\wkeq \Qsst\KGL$ or not.

Finally we mention that a motivic extension of Friedlander-Walker's construction (without products) has also appeared recently in \cite{KrishnaPark}.

An outline of this paper is as follows. In Section \ref{sec:motivic} we fix notations and recall some basic facts about motivic homotopy theory. In Section \ref{sec:fibrant} we construct  the motivic version of Shipley's detection functor and use it to produce a lax symmetric monoidal fibrant replacement functor. The motivic version of Friedlander-Walker's construction is defined in Section \ref{sec:sst} and some first general properties are recorded. In Section \ref{sec:ex} we record some examples of this construction and some fundamental properties of the motivic Friedlander-Walker construction applied to nice motivic spectra. Finally in Section \ref{sec:bott} we conclude with some applications to semi-topological cobordism.

\textbf{Notation:} We write $\Sch/\C$ for the category of separated schemes of finite type over $\C$ and $\Sm/\C$ for the  category of smooth schemes over $\C$. Our indexing convention for the motivic spheres is the standard one, $S^{p,q} = S^{p-q}\wedge \G_{m}^{\wedge q}$.
For a motivic spectrum $E$, we write $\ppsh_{p,q}E:\sm/\C^{op}\to \ab$ for the presheaf of abelian groups
$\ppsh_{p,q}E(X)= [S^{p,q}\wedge X_+, E]_{\SH(\C)}$.

\section{Preliminaries}\label{sec:motivic}
 In this section we recall the definitions and properties of the models which we use for the stable motivic homotopy category. In this section and the next $S$ is a Noetherian base scheme of finite Krull dimension. Our motivic suspension coordinate is any flasque cofibrant motivic space $T$, which is isomorphic to $(\P^{1},\infty)$ in $\hb(S)$. Standard choices for $T$ are $S^{1}\wedge \G_{m}$, $\A^1/\A^1-0$, and $\P^1$.

\subsection{Motivic model structures}
A \textit{based motivic space} $F$ on $S$ is a presheaf of based simplicial sets 
$F:\sm/S^{op}\to \ssetb$ and we write $\MMb(S)$ for the category of based motivic spaces over $S$. The \textit{motivic model structure} on $\MMb(S)$ which we use in this paper is the Bousfield localization of the global flasque model structure at Nisnevich local and $\A^{1}$-equivalences, which is a cellular and proper model category. Write $\hb(S)$ for its homotopy category (which is equivalent to Morel-Voevodsky's homotopy category). We refer to \cite{I:flasque} for details.

%
%
%

In general if $K$ is a based motivic space $K$, we write $\ihom(K,-)$ for the right adjoint of $K\wedge -$. We write $\Sigma_{T}F = T\wedge F$ and $\Omega_{T}F = \ihom(T,F)$.

A $T$-spectrum consists of a based motivic spaces $E =(E_{0}, E_{1},  \cdots)$ together with structure maps $\sigma_{i}: E_{i}\wedge T \to E_{i+1}$. A map of $T$-spectra $E\to E'$ consists of maps $E_{i}\to E_{i}'$ compatible with the structure maps. 
Write $\spt_{T}(S)$ for the category of $T$-spectra. As $T$ is cofibrant and $\MMb(S)$ is proper, cellular, and cofibrantly generated, Hovey's machinery \cite{Hovey:Spt} yields a stable model structure on $\spt_{T}(\mcal{C})$,
which is again proper, cellular, and cofibrantly generated.
A spectrum $E$ is fibrant in $\spt_{T}(S)$ if it is an $\Omega_T$-spectra: each $E_{i}$ is motivic fibrant and $E_{i}\to \Omega_{T}E_{i+1}$ is a motivic  equivalence.

The identity on $\spt_{T}(S)$ gives a left Quillen equivalence between this model structure Jardine's model structure \cite{Jar:motspt}, formed using the injective model structure on based motivic spaces.

A symmetric $T$-spectrum on $S$ consists of a $T$-spectrum $E$ such that each $E_{n}$ has a $\Sigma_{n}$-action and that every map $\sigma^{r}:X_{n}\wedge T^{r}\to X_{n+r}$ obtained as the 
iteration of the structure maps is $\Sigma_n\times \Sigma_k$-equivariant.
A map of symmetric spectra $E\to F$ is a map of spectra such that $E_{i}\to F_{i}$  is $\Sigma_i$-equivariant.
We write $\sspt_T(S)$ for the category of symmetric spectra.

Hovey's machinery \cite{Hovey:Spt} applies to produce a stable model structure on $\sspt_{T}(S)$. This model structure is proper, cellular, and monoidal.
By \cite[Theorem 8.8]{Hovey:Spt} a symmetric $T$-spectrum $E$ is fibrant if and only if it is an $\Omega_T$-spectra.

There is a Quillen equivalences functor $V:\spt_{T}(S)\rightleftarrows \sspt_{T}(S):U$, where $U$ forgets the extra structure of a symmetric spectrum. 
The distinction between equivalences in these two categories is important to keep in mind: a stable equivalence in $\spt_{T}^{\Sigma}(S)$ need not induce an isomorphism on naive stable motivic homotopy groups. This will be important in Section \ref{sec:fibrant}. 

\begin{definition}
\begin{enumerate}
\item A map $X\to Y$ of symmetric spectra is said to be a \textit{$U$-equivalence} if $U(X)\to U(Y)$ is a stable equivalence in $\spt_{T}(S)$.
\item A symmetric spectrum $X$ is \textit{semistable} if $X\to X^{\fib}$ is a $U$-equivalence, where $(-)^{\fib}$ is a stably fibrant replacement. 
\end{enumerate}
\end{definition}

The utility of this notion is that when $X$ is semistable, the underlying spectrum $UX$ agrees with the value of the total derived functor $(\R U)(X)$.

We finish by describing a set of generating trivial cofibrations for the stable structure, which will be used in the next section. The category $\sspt_{T}(S)$ has a projective model structure: a map $f:E\to E'$ is a weak equivalence or fibration if and only if each $E_{i}\to E'_{i}$ is one in $\MMb(S)$. 
Factor the map $F_{n+1}(X\wedge T) \to F_{n}(X)$ as a projective cofibration $\lambda_{n}^X:F_{n+1}(X\wedge T)\to C_{n}(X)$ 
followed by a trivial level fibration $C_{n}(X)\to F_{n}(X)$ (where $F_{n}:\MMb(S)\to \sspt_T(S)$ is the left adjoint to the functor $E\mapsto E_{n}$). 
Let $J^{mot}$  be a set of generating trivial cofibrations for $\MMb(S)$ whose domains are small relative to the generating cofibrations. Then $J^{proj} = \cup_{n} F_{n}J^{mot}$ is a set of generating trivial cofibrations for the projective level model structure on $\sspt_{T}(S)$   by \cite[Section 8]{Hovey:Spt}.
Recall that the pushout product $i\,\Box \,j$ of two morphisms $i:A\to X$ and $j:B\to Y$ is the morphism $i\,\Box\, j: A\wedge Y \coprod_{A\wedge B} X\wedge B \to X\wedge Y$.
Define 
\begin{equation}\label{eqn:jinj}
J^{\Sigma} = J^{proj} \cup \{\lambda_{n}^{X_+}\,\Box\left(\partial\Delta^{k}_{+}\to \Delta^{k}_{+}\right)|\, X\in \sm/S, \,n\geq 0,\,k\geq 0\}.
\end{equation}

Given a set of maps $A$ in a category $\mcal{C}$, an \textit{$A$-injective} is defined to be a map which has the right lifting property with respect to maps in $A$ and an \textit{$A$-cofibration} is a map which has the left lifting property with respect to $A$-injectives.

\begin{proposition}
The set $J^{\Sigma}$ forms a set of generating trivial cofibrations for the stable model structure on $\sspt_{T}(S)$. 
\end{proposition}
\begin{proof}
We have to show that the $J^{\Sigma}$-cofibrations are the trivial stable cofibrations. 
Elements of $J^{\Sigma}$ are trivial stable cofibrations. A $J^{\Sigma}$-cofibration is a
stable trivial cofibration because it is retract of a relative $J^{\Sigma}$-cell complex.

Now suppose that $f:A\to B$ is a stable trivial cofibration. 
The domains of $J^{\Sigma}$ are small relative to cofibrations, in particular we may apply the small object argument, \cite[Proposition 10.5.16]{Hir:loc} to factor $f$ as the composition of a $J^{\Sigma}$-cofibration $i:A\to C$ followed by a $J^{\Sigma}$-injective $p:C\to B$. As both $f$ and $i$ are stable equivalences, so is $p$. By Lemma \ref{lem:jinj} below, $p$ is a trivial level fibration. This implies that $f$ has the left lifting property with respect to $p$ which in turn implies that $f$ is a retract of $i$, and hence a $J^{\Sigma}$-cofibration.
\end{proof}

\begin{lemma}\label{lem:jinj}
 Suppose that $p:E\to D$ is a $J^{\Sigma}$-injective and a stable equivalence. Then it is a level trivial fibration.
\end{lemma}
\begin{proof}
 As it has the right lifting property with respect to $J^{proj}$ it is a level fibration. Consider the fiber $H$ of $p$. As $J^{\Sigma}$-injectives are closed under pull-back, the map $H\to *$ is $J^{\Sigma}$-injective. In particular it has the right lifting property with respect to every $\lambda_{n}^{X_+}\,\Box\left(\partial\Delta^{k}_{+}\to \Delta^{k}_{+}\right)$ which implies that $\shom(X, H_{n}) \to \shom(X, \Omega_{T}H_{n+1})$ is a weak equivalence of simplicial sets for every $X$ in $\sm/S$. In particular $H$ is a stably fibrant motivic spectrum. The map $E/H\to D$ is a stable equivalence, and thus $E\to E/H$ is a stable equivalence. This implies that $H\to *$ is a stable equivalence, and thus each $H_{n}\to *$ is  a motivic equivalence which implies the result. 
\end{proof}

\subsection{Simplicial objects}
In this section we record a few useful facts regarding simplicial motivic symmetric spectra and their realizations for which we don't have a reference directly applying to our setting.

\begin{proposition}\label{prop:hopush}
Suppose that 
$$
\xymatrix{
A \ar[r]\ar[d] & B\ar[d]\\ 
X \ar[r] & Y
}
$$
is a pushout in $\spt^{\Sigma}_{T}(S)$ where $A\to X$ is a monomorphism. Write $P$ for the homotopy pushout in $\spt^{\Sigma}_{T}(S)$ (resp.~ in $\spt_{T}(S)$) of $X\leftarrow A \to B$. Then 
 the map $P\to Y$ 
is a stable equivalence (resp.~$U$-equivalence).
\end{proposition}
\begin{proof}
Recall that the homotopy pushout is the ordinary pushout of the diagram $(X^{\cof})' \leftarrow A^{\cof} \to B^{\cof}$, where $(-)^{\cof}$ is a stable cofibrant replacement functor and $A^{\cof}\to (X^{\cof})'\to X^{\cof}$ is a factorization of $A^{\cof}\to X^{\cof}$ as a stable cofibration followed by a level trivial fibration.  
The proposition thus follows immediately from Lemmas \ref{lem:lem1} and \ref{lem:lem2} below.
\end{proof}

\begin{lemma}\label{lem:lem1}
Let 
$$
\xymatrix{
A \ar[r]^{f}\ar[d] & B\ar[d]\\ 
X \ar[r]^{f'} & Y
}
$$
be pushout in $\spt^{\Sigma}_{T}(S)$ where $A\to X$ is a monomorphism. If $f:A\to B$ is a stable equivalence (resp.~$U$ equivalence) then so is $X\to Y$.
\end{lemma}
\begin{proof}
Pushouts along monomorphisms preserve levelwise equivalences. This follows from the left properness of the injective motivic model structure, that injective cofibrations of motivic spaces are the monomorphisms,  and that equivalences in $\MMb(S)$ agree with those in the injective motivic structure.

Let $A\to B$ is a stable equivalence and write $A\to B'\to B$ for a factorization in terms of a stable trivial cofibration followed by a level trivial fibration $B'\to B$. 
As $X\to X\coprod_{A}B'$ is the pushout of a trivial stable cofibration, it is a stable equivalence. The previous observation implies that $X\coprod_{A}B'\to Y$ is a level equivalence.

Apply the functor $U$ and the argument in the previous paragraph to get the statement concerning $U$-equivalences.
\end{proof}

\begin{lemma}\label{lem:lem2}
 Suppose that 
$$
\xymatrix{
C_{1} \ar[d] & A_{1} \ar@{_{(}->}[l]_-{i_1}\ar[d]\ar[r] & B_{1} \ar[d] \\
C_{2} & A_{2} \ar@{_{(}->}[l]_-{i_{2}}\ar[r] & B_{2}
}
$$
is a commutative diagram in $\spt_{T}(S)$, where the vertical arrows are stable equivalences (resp.~$U$-equivalences) and $i_1$, $i_{2}$ are monomorphisms. Then the induced map  $C_{1}\coprod_{A_{1}}B_{1}\to C_{2}\coprod_{A_{2}} B_{2}$ is a stable equivalence (resp.~$U$-equivalence). 
\end{lemma}
\begin{proof}
The previous lemma reduces the statement to the case when the right horizontal maps are cofibrations (consider a factorization of $A_{i}\to B_{i}$ into a stable cofibration followed by a level trivial fibration). The statement is then a standard fact about left proper model categories.
\end{proof}

\begin{definition}
The \textit{diagonal} $|W|$ of a simplicial object $d\mapsto W(d)$ in $\spt_{T}^{\Sigma}(S)$ 
 is the levelwise, schemewise diagonal. That is,
$|W|_{i} = \diag (d\mapsto W(d)_{i})$, is the schemewise application of the usual diagonal functor for bisimplicial sets. The  structure maps are given by
$\diag\left(W(d)_{i}\right)\wedge T = \diag \left(W(d)_{i}\wedge T\right) 
\to \diag\left(W(d)_{i+1}\right)$.
\end{definition}

We have a coequalizer diagram
$\coprod_{m\to n}W(n)\wedge \Delta^m_+ \rightrightarrows \coprod_{n}W(n)\wedge \Delta^{n}_+ 
\to |W|$ in $\spt_{T}^{\Sigma}(S)$. Filter $|W|$ by  
$$
|W|^{(n)} = \mathrm{Image}\left(\coprod_{k\leq n}W(k)\wedge\Delta^{k}_+ \to |W|\right) \subseteq |W|
$$ 
and $|W|^{(n)} = *$ for $n<0$.
Let $s_{[r]}W(p) = \cup_{0\leq i\leq r}W(p)\subseteq W(p+1)$. We obtain two pushout squares 
\begin{equation}\label{eqn:PO1}
\xymatrix{
s_{[r]}W(p-1) \ar@{^{(}->}[r]\ar[d] & W(p) \ar[d] \\
 s_{[r]}W(p)\ar[r] & s_{[r+1]}W(p),
}
\end{equation}
and
\begin{equation}\label{eqn:PO2}
\xymatrix{
{s_{[p]}W(p)\wedge \Delta^{p+1}_+\coprod_{s_{[p]}W(p)\wedge \partial\Delta^{p+1}_{+}}
W(p+1)\wedge\partial \Delta^{p+1}_+  }\ar@{^{(}->}[r]\ar[d] &
W(p+1)\wedge \Delta^{p+1}_{+} \ar[d]\\
*+[r]{|W|^{(p)}}\ar[r] & *+[r]{|W|^{(p+1)}.}
}
\end{equation}
A standard inductive argument yields the following.
\begin{theorem}\label{thm:diag}
 Let $X(d)\to Y(d)$ be a map of  simplicial objects in $\spt_{T}^{\Sigma}(S)$ which is a level equivalence (resp.~$U$-equivalence, resp.~stable equivalence) for each $d$. Then the induced map $|X|\to |Y|$ is a level equivalence, (resp.~$U$-equivalence, resp.~stable equivalence).
\end{theorem}

There are two points of view commonly taken regarding homotopy colimits in a model category (which ultimately yield the same result in the homotopy category). One is that homotopy colimits are defined as the derived functors of colimit and the other is that it is the result of applying a direct formula such as the Bousfield-Kan formula \cite{BK}. The former approach is taken in \cite{CS:diagrams} and the latter in \cite{H:loc}. For our purposes it is convenient to define a homotopy colimit  to be the result of applying an explicit formula, namely the bar construction. 
\begin{definition}\label{def:hocolim}
Let $J$ be a small category. The homotopy colimit of a functor 
 $F:J\to \spt_{T}^{\Sigma}(S)$ is the realization $|B(\ast, J,F)|$ of the simplicial object 
$$
B(*,J,F)_{n} = \coprod_{j_{n}\to\cdots\to j_{0}}F(j_{n}).
$$
\end{definition}

Note that $(\hocolim_{J}F)_{n}(X) = \hocolim_{J} (F_{n}(X))$. 

\begin{lemma}\label{lem:hocolim}
Let $F,G:J\to \spt_{T}^{\Sigma}(S)$ be two functors and $F\to G$ a natural transformation such that $F(j)\to G(j)$ is a levelwise equivalence (resp.~$U$-equivalence, resp.~stable equivalence). Then $\hocolim_{J}F\to \hocolim_{J}G$ is also a levelwise equivalence (resp.~$U$-equivalence, resp.~stable equivalence).
\end{lemma}
\begin{proof}
 This follows from Theorem \ref{thm:diag} together with the fact that coproducts preserve levelwise equivalences (resp.~$U$-equivalences, resp.~stable equivalences).
\end{proof}


Next we record a spectral sequence relating the homotopy presheaves of a simplicial motivic spectrum to those of its diagonal. This will be useful for out applications in the later sections. The construction presented here is a motivic translation of Jardine's construction \cite{Jar:genet} in the case of ordinary simplicial spectra.

 A motivic spectrum $X$ is said to be \textit{compact} provided the functor $[X,- ]_{\SH(S)}$ commutes with filtered colimits. 
\begin{theorem}\label{thm:sspseq}
 Let $d\mapsto E(d)$ be a simplicial object in $\spt_{T}^{\Sigma}(S)$. 
For any compact motivic spectrum $X$, we have a convergent spectral sequence
$$
E^{2}_{p,q} = H_{p}\big(d\mapsto [\Sigma^{q,t}X,\, E(d)]_{\SH(S)}\big) \Longrightarrow [\Sigma^{p+q,t}X,\,|E|]_{\SH(S)}.
$$
\end{theorem}
\begin{proof}
Jardine's construction given in \cite[Proposition 4.21]{Jar:genet} in the case of ordinary simplicial spectra applies in the motivic setting. We present the main points and refer to loc.~cit.~ for complete details.

Write $|E|^{(p/p-1)} = |E|^{(p)}/|E|^{(p-1)}$. We have homotopy cofiber sequences 
$$
|E|^{(p-1)}\to |E|^{(p)}\to |E|^{(p/p-1)}\xrightarrow{\partial} \Sigma^{1,0}|E|^{(p-1)}.
$$
Setting $D^{1}_{p,q}= [\Sigma^{p+q,t}X, |E|^{(p)}]$ and $E^{1}_{p,q} = [\Sigma^{p+q,t}X, |E|^{(p/p-1)}]$ yields an exact couple and hence an associated spectral sequence. 
The target of this spectral sequence is $D^{\infty}_{n} = \colim_{p\to \infty}D^{1}_{p,n-p} = \colim_{p\to \infty}[\Sigma^{n,t}X, |E|^{(p)}] =  [\Sigma^{n,t}X, |E|]$. Since $D^{1}_{p,q} = 0$ for $p<0$. the spectral sequence is  convergent.

The differential $E^{1}_{p,q} \to E^{1}_{p-1,q}$ is induced by applying $[\Sigma^{p+q,t}X, -]$ to the composite map 
$|E|^{(p/p-1)}\xrightarrow{\partial} \Sigma^{1,0}|E|^{(p-1)}
\xrightarrow{\phi}\Sigma^{1,0}|E|^{(p-1/p-2)}$.
An examination, similar to that in \cite[p. 102]{Jar:genet}, of the pushout squares (\ref{eqn:PO1}) and (\ref{eqn:PO2}) appearing before Theorem \ref{thm:diag} shows that there is a natural isomorphism
\begin{equation}\label{eqn:pp-1}
|E|^{(p/p-1)} \iso (E(p)/s_{[p-1]}E(p-1))\wedge S^{p}.
\end{equation}
For a simplicial abelian group $A_{*}$, define as usual $s_{[r]}A$ to be the subgroup generated by the images of the $s_{i}$, $0\leq i\leq r$.
 The argument of \cite[Lemma 4.15]{Jar:genet} applies here and shows that the inclusion 
$s_{[r]}[\Sigma^{n,t}X, E(p)]\subseteq [\Sigma^{n,t}X, E(p)]$ has a natural factorization 
$$
\xymatrix{
s_{[r]}[\Sigma^{n,t}X, E(p)]\ar@{^{(}->}[r] \ar[d]^{\iso} & [\Sigma^{n,t}X, E(p)]\\
[\Sigma^{p+q,t}X, s_{[r]}E(p)]. \ar[ur]_{i_*} & 
}
$$
This yields the natural isomorphism 
\begin{equation}\label{eqn:e1}
E^{1}_{p,q}= [\Sigma^{q,t}X, E(p)]/s_{[p-1]}[\Sigma^{n,t}X, E(p)] = \mcal{N}_p\big(d\mapsto[\Sigma^{q,t}X, E(d)]\big),
\end{equation}
where  $\mcal{N}_*A$ is the normalized chain complex associated to a simplicial abelian group.

It remains to identify the homology of $E^{1}_{*,q}$ with that of the simplicial abelian group $d\mapsto[\Sigma^{q,t}X, E(d)]$. First we consider the comparison diagram
$$
\xymatrix{
E(p)\wedge \sk_{p-2}\Delta^{p}_{+} \ar@{^{(}->}[r]\ar[d] & 
E(p) \wedge \partial \Delta^p_+  \ar@{^{(}->}[r]\ar[d] & E(p)\wedge \Delta^{p}_+ \ar[d] \\
*+[r]{|E|^{(p-2)}} \ar@{^{(}->}[r]&  *+[r]{|E|^{(p-1)}} \ar@{^{(}->}[r] & *+[r]{|E|^{(p)}.}
}
$$
This yields the commutative diagram
\begin{equation}\label{eqn:bigd}
\xymatrix{
E(p)\wedge S^{p} \ar[r]^-{\partial}\ar[d] & 
\Sigma^{1,0} E(p) \wedge \partial \Delta^p_+  \ar[r]^-{\phi}\ar[d] & \Sigma^{0,1}E(p)\wedge \partial\Delta^{p}_+/\sk_{p-2}\Delta^{p}_{+} \ar[d]\\
*+[r]{|E|^{(p/p-1)}} \ar[r]^-{\partial}&  *+[r]{|E|^{(p-1)} }\ar[r]^-{\phi} & *+[r]{|E|^{(p-1/p-2)} .}
}
\end{equation}
The simplicial set $\partial\Delta^{p}_+/\sk_{p-2}\Delta^{p}_{+}$ is a $(p+1)$-fold wedge of copies of $S^{p-1}$, with the $i$th summand corresponding to the collapsing of all but the $i$th face of $\Delta^{p}$ to a point. Applying $[\Sigma^{p+q,t}X, -]$ and using this identification, 
the top row becomes $[\Sigma^{q,t}X, E(p)] \to \oplus_{0\leq i \leq p} [\Sigma^{q,t}X, E(p)]$. The map into the $i$th component is multiplication by the degree of the map 
$S^{p}\to S^{p}$ which is the composite of 
 $S^{p}\xrightarrow{\partial} \Sigma\partial\Delta^p_+ \xrightarrow{\phi} \Sigma\partial\Delta^{p}_+/\sk_{p-2}\Delta^{p}_{+}$ with the projection $\partial\Delta^{p}_+/\sk_{p-2}\Delta^{p}_{+}\to S^{p-1}$ to the $i$th component. 
In \cite[Lemma 4.9]{Jar:genet} it is shown this is map has degree $(-1)^{i}$.

The composition $E(p)\wedge \Delta^{p-1}_+ \to E(p)\wedge \Delta^{p}_+\to |E|$ induced by $d^i$ and the composition induced by $d_{i}$, 
$E(p)\wedge \Delta^{p-1}_+\to E(p-1)\wedge \Delta^{p-1}_+ \to |E|$  coincide. One concludes that there is a commutative diagram
\begin{equation}\label{eqn:iddif}
\xymatrix@C-1pc{
\oplus_i [\Sigma^{q,t}X, E(p)] \ar[r]^-{\iso}\ar[d]_{\sum(-1)^id_i} & [\Sigma^{p+q-1,t}X, E(p)\wedge \partial\Delta^{p}_+/\sk_{p-2}\Delta^{p}_+]\ar[dr] &  \\
[\Sigma^{q,t}X, E(p-1)] \ar[r] & [\Sigma^{p+q-1,t}X, E(p)\wedge S^{p-1}] \ar[r] & E_1^{p-1,q}.
}
\end{equation}
Combining the diagrams (\ref{eqn:pp-1}), (\ref{eqn:e1}), (\ref{eqn:bigd}), and (\ref{eqn:iddif}) easily yields the desired identification of the $E^{2}$-term.
\end{proof}

\section{Shipley's construction in motivic homotopy theory}\label{sec:fibrant}
In this section we show that Shipley's detection functor \cite{Shipley:THH} for symmetric spectra in simplicial sets admits a generalization to  motivic symmetric spectra. As a result, we obtain 
a lax symmetric monoidal fibrant replacement functor on motivic symmetric spectra, which will be used in the next section. The construction in this section follows Shipley's strategy in loc.~cit.~ and ultimately the key point is to relate certain  free $\II$-diagrams in motivic symmetric spectra with free motivic symmetric spectra.  At this point, Shipley makes use of spacewise connectivity arguments based on the Blakers-Massey and the Freudenthal Suspension theorems which do not apply in the motivic setting. Fortunately, these spacewise arguments are readily replaced with a spectrum level argument that does apply to motivic symmetric spectra.


As observed in \cite[Lemma 2.2]{RSO:strict}, there is a lax symmetric monoidal fibrant replacement functor $\LL'$ on $\MMb(S)$. This in turn induces a lax symmetric monoidal level-wise fibrant replacement on $\sspt_{T}(S)$, which we again denote by $\LL'$.

We write $\II$ for the category whose objects are the finite sets $\nn=\{1,\ldots,n\}$ and whose morphisms are injections. 
Let $K$ be a based motivic space. Recall that the free symmetric spectrum  $F_{m}K$ is defined by
$$
F_{m}(K)_{n}=(\Sigma_{n})_+\wedge_{\Sigma_{n-m}} (K\wedge T^{n-m})
$$ 
if $n\geq m$ and otherwise $(F_mK)_n= *$.  
We have an isomorphism  of $\Sigma_{n}$-sets, ${\II}(\mm,\nn)\iso \Sigma_{n}/\Sigma_{n-m}$. It follows that $F_{m}K$ admits following the alternate description
 $(F_mK)_n= {\II}(\mm,\nn)_+\wedge K\wedge T^{n-m}$. This second description of the free spectra $F_{m}K$ is important for Lemma \ref{lem:mainpt}, which is the key to showing that Shipley's detection functor $D$, defined below, preserves stable equivalences. 
\begin{definition}
Let $X$ be a symmetric motivic spectrum. Define a functor $\mcal{D}_{X}:\II\to \spt_{T}^{\Sigma}(S)$ as follows. On objects, $\mcal{D}_{X}(\nn) = \Omega_{T}^{n}\LL'F_{0}X_{n}$. Every morphism in $\II$ is a composition of a standard inclusion and an isomorphism. 
For an isomorphism $\alpha:\mm\iso\mm$, define $\mcal{D}_{X}(\alpha)$ to be the composite of the conjugation action on the loop coordinates together with the action on $X_{m}$. 
For a standard inclusion $\iota:\nn\subseteq \mathbf{n+k}$, define $\mcal{D}_{X}(\iota):\Omega_{T}^{n}\LL'F_{0}X_{n}\to \Omega_{T}^{n+k}\LL'F_{0}X_{n+k}$ to be  $\Omega_{T}^{n}\LL'$ applied to the composite induced by the structure map
$$
 F_{0}X_{n} \to F_{0}\Omega_{T}^{k} X_{n+k} \to \Omega_{T}^{k}F_{0}X_{n+k}
$$ 
followed by the natural map $\LL'\Omega_{T}^{k}F_{0}X_{n+k}\to \Omega_{T}^{k}\LL'F_{0}X_{n+k}$.

\textit{Shipley's detection functor} 
$D:\sspt_{T}(S)\to \sspt_{T}(S)$
is defined by
$$
DX = \hocolim_{\II} \mcal{D}_{X},
$$
where $\hocolim$ is defined via a bar construction model (see Definition \ref{def:hocolim}).
\end{definition}

Write $\N\subseteq \II$ for the subcategory whose objects are  $\nn$ and nonidentity morphisms are the standard inclusions. Colimit diagrams over $\N$ are homotopy colimit diagrams.
 Note that $(\colim_{\N}\mcal{D}_{X})_n = \colim_k \Omega^k_T\LL'\Sigma^n_TX_k$, where the transition maps are induced by the structure maps of $X$, which is a model for the fibrant replacement in $\spt_{T}(S)$. 


Write $\omega$ for the set of natural numbers and $\II_{\omega}$
for the category whose objects are the finite sets $\nn$ together with  $\omega$ and whose morphisms are injections. Given a functor ${F}:\II\to \ssetb$ write 
$\Lhk{F}:\II_{\omega}\to \ssetb$ for the left homotopy Kan extension of ${F}$. 
Let $M\subseteq \II_{\omega}$ be the full subcategory containing the object $\omega$. 
Shipley attributes the following useful proposition to J. Smith, which we restate in the setting of motivic spectra for convenience.
\begin{proposition}[{\cite[Proposition 2.2.9]{Shipley:THH}}]\label{prop:smith}
 Let ${F}:\II\to \sspt_{T}(S)$ be a functor. Then there are natural weak equivalences
\begin{enumerate}
\item $\displaystyle{\hocolim_{M}\Lhk{F}(\omega) \xrightarrow{\wkeq} \hocolim_{\II_{\omega}} {F}}$, and
\item $\displaystyle{\hocolim_{\N}{F}\xrightarrow{\wkeq} \Lhk{F}(\omega)}$.
\end{enumerate}
\end{proposition}
\begin{proof}
Homotopy colimits of motivic spectra are computed level and objectwise and so
 the result proved in loc.~cit.~ for functors $F:\II\to \ssetb$ applies here as well.  
\end{proof}

For each $\mm$ and motivic spectrum $W$, define $\free_{m}(W):\II\to \spt_{T}^{\Sigma}(S)$ by 
$$
\free_{m}(W)(\nn)= {\II}(\mm,\nn)_+\wedge W.
$$ 
This is a freely generated $\II$-diagram and the functor $W\mapsto \free_{m}(W)$ is left adjoint to evaluation at $\mm$.

\begin{lemma}\label{lem:mainpt}
 Let $K$ be a based motivic space. There is a natural map of motivic symmetric spectra 
$$
\Omega^{m}_{T}\mcal{L}'F_{0}K\xrightarrow{\wkeq} DF_{m}K
$$
which is a $U$-equivalence.
\end{lemma}
\begin{proof}
The inclusion $K\to \II(\mm,\mm)_{+}\wedge K$ as the wedge summand corresponding to the identity induces the map $\Omega^{m}_{T}\LL'F_{0}K\to \Omega^{m}_{T}\LL'F_{0}(\II(\mm,\mm)_{+}\wedge K)= \mcal{D}_{F_{m}K}(\mm)$.
The adjoint is a map of $\II$-diagrams 
$$
\free_{m}(\Omega^{m}_{T}\LL'F_{0}K)\to \mcal{D}_{F_{m}K}.
$$ 
For $\nn$ in $\II$, with $n\geq m$, this map is a composition 
%
%
%
\begin{align*}
 {\II}(\mm,\nn)_+  \wedge \Omega^m_{T}\LL'( F_{0}K)  & \to  
\Omega^{m}_{T}\LL'(F_{0}(\II(\mm,\nn)_+\wedge K)) \\
& \to\Omega^{m}_{T}\Omega_{T}^{n-m}\LL'F_{0}(\II(\mm,\nn)_+\wedge K\wedge T^{n-m}).
\end{align*}
The first map is the map which on the factor indexed by $\alpha:\mm\to\nn$ is induced by
by including $K\to \II(\mm,\nn)_{+}\wedge K$ as the $\alpha$th wedge summand. 
The second map is obtained from the natural transformation $\id\to \Omega_{T}^{n-m}(-\wedge T^{n-m})$.

The first map is an equivalence on underlying $T$-spectra as finite products and coproducts agree in $\mathrm{Ho}(\spt_{T}(S))$. 
The second map is  isomorphic in $\mathrm{Ho}(\spt_{T}(S))$ to an iterated application of the transformation $\id\to \Omega_{T}\Sigma_{T}$, thus a $U$-equivalence.

Now, since the map $\free(\Omega^{m}_{T}\mcal{L}'F_{0}K)(\nn)\to \mcal{D}_{F_{m}K}(\nn)$ is a $U$-equivalence for all $n\geq m$, the induced map on homotopy colimits is a $U$-equivalence. But the homotopy colimit over a free diagram is equivalent to the colimit  and so we obtain the desired $U$-equivalences
$\Omega^{m}_{T}\mcal{L}'F_{0}K \xrightarrow{\wkeq} \hocolim_{\II}\free_{m}(\Omega^{m}_{T}\mcal{L}'F_{0}K)\xrightarrow{\wkeq} D_{F_{m}K}$.
\end{proof}

\begin{theorem}\label{thm:Dprop}
The functor $D$ has the following properties.
\begin{enumerate}
\item For any $X$, $DX$ is semi-stable.
\item If $f:X\to Y$ is a $U$-equivalence, then $DX\to DY$ is a level-equivalence. 
\item If $f$  is a stable equivalence then $Df$ is a $U$-equivalence.
\item There is a natural zig-zag of $U$-equivalences relating $X^{\fib}$ and $DX$. 
\end{enumerate}
\end{theorem}
\textit{Proof of (1).}
 The $n$th level of $DX$ is $\hocolim_{k\in \II}\Omega_{T}^{k}\mcal{L}'\Sigma_{T}^{n}X_{k}$, with the $\Sigma_{n}$-action coming from that on the coordinates of $\Sigma_{T}^{n}$. As the action of an even permutation on a sphere is trivial in $\hb(k)$, it follows from 
\cite[Proposition 3.2]{RSO:strict} that $DX$ is semi-stable.

\textit{Proof of (2).}
This follows immediately from Proposition \ref{prop:smith} together with the fact that $\hocolim_{\N}\mcal{D}_{X}$ is a fibrant replacement functor on $\spt_{T}(S)$.

\textit{Proof of (3).}
Suppose that $X\to Y$ is a stable equivalence. We may factor it as a stable trivial cofibration followed by a level trivial fibration. As $D$ preserves level equivalences, it is enough to show that $D$ takes stable trivial cofibrations to $U$-equivalences. 
Trivial cofibrations are obtained as retracts of sequential colimits of pushouts of generating trivial cofibrations. Retracts and sequential colimits preserve $U$-equivalences. This means that we need to show that $U$ sends pushouts along cofibrations to homotopy pushouts and sends generating trivial cofibrations to $U$-equivalences.

Let $Y$ be the pushout of $X\leftarrow A \to B$, where $A\to X$ is a cofibration.  For each $n$, the diagram  
$$
\xymatrix{
F_{0}A_{n} \ar[r]\ar[d] & F_{0}B_{n} \ar[d] \\
F_{0}X_{n} \ar[r] &  F_{0}Y_{n}
}
$$ 
is a homotopy pushout (pushouts are formed levelwise, thus $Y_{n}$ is the pushout in motivic spaces of $X_{n}\leftarrow A_{n}\rightarrow B_{n}$, $F_{0}$ preserves pushouts, the resulting square is a homotopy pushout by Proposition \ref{prop:hopush}). The functor $\Omega_{T}^{n}\mcal{L}'$ preserves homotopy pushouts. Since homotopy colimits commute, we conclude that 
$$
\xymatrix{
DA \ar[r]\ar[d] & DB \ar[d] \\
DX \ar[r] &  DY
} 
$$
is a homotopy pushout.

Recall the set $J^{\Sigma}$, see (\ref{eqn:jinj}), of generating trivial cofibrations. Some of the maps in this set are level equivalences, which are preserved by $D$. 
If $K$ is a motivic space, then Lemma \ref{lem:mainpt} implies that 
$D(F_{n+1}(K\wedge T)) \to D(F_{n}(K))$ is a $U$-equivalence. From this and that $D$ is compatible with homotopy pushouts, it follows that $D(\lambda_{n}^{X_+}\,\Box\,g)$ is a $U$-equivalence.

\textit{Proof of (4).} Define $MX = \hocolim_{\II}\Omega_T^n\LL'\shift_n X$.
We have a natural inclusion $X\to MX$ and the map $F_{0}X_n\to \shift_n X$ induces the natural map $DX\to MX$.  
 We thus have a natural diagram
$$
\xymatrix{
X \ar[r]\ar[d] & MX \ar[d] & DX \ar[l]\ar[d] \\
X^{\fib} \ar[r] & MX^{\fib} & DX^{\fib}. \ar[l]
}
$$ 
The right vertical arrow is a $U$-equivalence by the previous part. Since $X^{\fib}$ is an $\Omega_{T}$-spectrum, it is not hard to see that the bottom arrows are level equivalences. This provides the desired zig-zag of $U$-equivalences relating $X^{\fib}$ and $DX$. 
\qed

\begin{theorem}\label{thm:laxfib}
The functor $D$ is lax symmetric monoidal and $\mcal{L}'D^{2}X$ is a lax symmetric monoidal fibrant replacement functor.
\end{theorem}
\begin{proof}
 Consider the symmetric monoidal product $+:\II\times \II \to \II$ defined on objects by $(\nn,\mm)\mapsto \nn+\mm$ and by the obvious formula on morphisms. 
Given spectra $X$ and $Y$, we have an $\II\times \II$ diagram $\mcal{D}_{X}\wedge \mcal{D}_{Y}$ defined by  
$$
(\mcal{D}_{X}\wedge \mcal{D}_{Y})(\nn,\mm) = \mcal{D}_{X}(\nn)\wedge \mcal{D}_{Y}(\mm) = \Omega_T^{n}\LL'F_{0}X_{n} \wedge \Omega_T^m\LL'F_{0}Y_m .
$$ 

We have a pairing $\mcal{D}_{X}\wedge \mcal{D}_{Y} \to (\mcal{D}_{X\wedge Y})\circ +$ given by
\begin{align*}
\Omega_T^{n}\LL'F_{0}X_{n} & \wedge \Omega_T^m\LL'F_{0}Y_m \to \Omega_T^{n+m}(\LL'F_{0}X_{n} \wedge \LL'F_{0}Y_m) \\
& \to \Omega^{n+m}_T\LL' F_{0}(X_{n}\wedge Y_{m}) \to \Omega^{n+m}_T\LL' F_{0}(X\wedge Y)_{n+m}.
\end{align*}
Taking homotopy colimits, we have the natural map $DX\wedge DY \to D(X\wedge Y)$ obtained from the composite on homotopy colimits
\begin{align*}
DX\wedge DY = \hocolim_{\II}\mcal{D}_{X}  & \wedge \hocolim_{\II}\mcal{D}_{Y} \to \hocolim_{\II\times\II}\mcal{D}_{X}\wedge\mcal{D}_{Y} \\
&  \to \hocolim_{\II\times \II}(\mcal{D}_{X\wedge Y}) \circ + \to \hocolim_{\II} \mcal{D}_{X\wedge Y} =D(X\wedge Y).
\end{align*}

That $\LL'D^2X$ is a fibrant model for $X$ follows from the previous theorem. 

\end{proof}

\section{Semi-topological cohomology theories}\label{sec:sst}
In this section we 
extend  Friedlander-Walker's 
``singular semi-topological'' construction to a motivic construction $\Qsst$ that is lax symmetric monoidal. As a consequence, monoids in motivic symmetric spectra are preserved by this construction and ultimately it produces motivic strict ring spectra representing morphic cohomology, semi-topological $K$-theory, and semi-topological cobordism.

\subsection{Motivic homotopy on singular schemes}
Friedlander-Walker's constructions and techniques apply to presheaves defined on all of $\sch/\C$. Because of this, it is useful to work with motivic spaces defined on all schemes, an approach which yields the same stable motivic homotopy category in the presence of resolution of singularities. Let $k$ be a field admitting resolution of singularities. We write $\MMb(\sch/k)$ for the category of presheaves 
$F:\sch/k^{op}\to \ssetb$ and equip it with the flasque, cdh-, $\A^{1}$-local model structure.
In this setting, we sometimes write $\MMb(\sm/k)$ for the model category previously denoted $\MMb(k)$ (i.e. the model category of motivic spaces on $\sm/k$ equipped with the flasque, Nisnevich, $\A^1$-local model structure).

From now on we will take $(\P^1,\infty)$ as our motivic suspension coordinate. The model categories of $\P^1$-spectra and symmetric $\P^1$-spectra on $\sch/k$ are defined in exactly the same way as in Section \ref{sec:motivic}. We denote these categories respectively by $\spt_{\P^1}(\sch/k)$ and $\sspt_{\P^1}(\sch/k)$.
Write $i:\sm/k\subseteq \sch/k$ for the inclusion of categories. 
The Quillen pair $i^*:\MMb(k) \rightleftarrows \MMb(\sch/k):i_*$ extends to a Quillen pair on spectra is a Quillen equivalence by a theorem of Voevodsky.

\begin{theorem}
Let $k$ be a field admitting resolution of singularities. Then 
$$
i^{*}:\sspt_{\P^1}(k) \rightleftarrows \sspt_{\P^1}(\sch/k): i_*
$$
and 
$$
i^{*}:\spt_{\P^1}(k) \rightleftarrows \spt_{\P^1}(\sch/k): i_*
$$
are Quillen equivalences. Moreover, $i^*$ is strong symmetric monoidal and there exists a lax symmetric monoidal fibrant replacement functor for $\sspt_{\P^1}(\sch/k)$.
\end{theorem}
\begin{proof}
The Quillen equivalences follow immediately from \cite{V:niscdh} because the model structures we use are all equivalent to the one used there. The second statement follows from the fact that $\sm/k\subseteq \sch/k$ is strong symmetric monoidal. The construction of a lax symmetric monoidal fibrant replacement functor in Section \ref{sec:fibrant} works equally well for $\spt_{\P^1}(\sch/k)$. 
\end{proof}

\begin{remark}
For a motivic spectrum $E$, we have a presheaf of abelian groups $\ppsh_{s,t}E$ on $\sm/k$, $U\mapsto [\Sigma^{s,t}U_{+}, E]_{\SH(k)}$. We will use the same notation for presheaf on $\sch/k$ defined by the same formula.
\end{remark}

%
%


\subsection{Topological realization}\label{sub:top}

The functor $\Sch/\C \to \ssetb$ which sends $X$ to $\sing X(\C)_{+}$ can be extended to a functor $\RRe_{\C}:\MMb(\Sch/\C)\to \ssetb$ by 
$$
\RRe_{\C}F = \colim_{(X\times \Delta^{n})_+\to F}(\sing X(\C) \times \Delta^{n})_{+}.
$$ 
This functor has a right adjoint  
defined by $\msing_{\C}(K)(X) = \shom(\sing X(\C), K)$ (where $\shom(-,-)$ is the simplicial set of maps).

\begin{remark}
Write $\RRe_{\C}':\MMb(\sm/\C)\to \btop$ for the ``usual'' topological realization functor defined by
$\RRe_{\C}'F = \colim_{(X\times\Delta^n)_+\to F}(X(\C)\times\Delta^{n}_{top})_+$.
This is related to the topological realization considered here  by 
$\sing\L\RRe_{\C}'F \wkeq \RRe_{\C}\L i^*F$. 
\end{remark}

\begin{proposition}
 The adjoint pair
$$
\RRe_{\C}:\MMb(\Sch/\C) \rightleftarrows \ssetb: \msing_{\C}
$$
is a Quillen adjunction. Moreover $\RRe_{\C}$ is a  strong monoidal functor.
\end{proposition}
\begin{proof}
The argument is similar to that given in \cite[Theorem A.23]{PPR:KGL}. The key point is that by standard facts about localizations of model categories it suffices to show that these are Quillen pairs on the global flasque model structure and the left adjoints send cdh-distinguished squares to homotopy pushouts  and send maps of the form $X_{+}\to (X\times \A^{1})_{+}$ to weak equivalences.

To show that the left adjoints are a Quillen pair on the global flasque model structure it suffices by Dugger's lemma \cite[Corollary A2]{dugger:simp} to show that the right adjoints of these pairs preserve trivial fibrations as well as fibrations between fibrant objects. 
This follows from the fact that $\Re_{\C}$  maps the generating cofibrations $I$ to cofibrations in $\ssetb$  maps generating trivial cofibrations $J$ to weak homotopy equivalences in $\ssetb$.

As $\A^{1}(\C)$ is contractible, $\Re_{\C}$ sends maps of the form $X_{+}\to (X\times \A^{1})_{+}$ to homotopy equivalences. That $\Re_{\C}$ sends Nisnevich distinguished squares to homotopy pushouts follows from the arguments in \cite[Section 5]{DI:hyp}.
Suppose that 
\begin{equation*}
\xymatrix{
B \ar[r]^{j}\ar[d] & Y \ar[d]^{p}\\
A \ar[r]^{i} & X
}
\end{equation*}
is a distinguished $cdh$-square. We have $X(\C) = A(\C) \coprod_{B(\C)} Y(\C)$ and  $Y(\C)$ can be triangulated so that  $B(\C)\subseteq Y(\C)$ is a subcomplex. 
This implies that $X(\C)$ is the homotopy pushout out $A(\C)\leftarrow B(\C) \rightarrow Y(\C)$. Since $\sing$ preserves homotopy colimits, see e.g. \cite[Proposition 18.9.12]{Hir:loc}, this implies the $\Re_{\C}$ sends distinguished square to homotopy pushouts.

The last statement is a simple consequence of the standard fact that there is a natural homeomorphism $(X\times Y)(\C) \iso X(\C)\times Y(\C)$.
\end{proof}

Write $\tilde{S}^2 = \RRe_{\C}(\P^{1}) = \sing \C\P^1$. There is a canonical weak equivalence $S^{2}\wkeq \tilde{S}^{2}$.
The categories of spectra $\spt_{\tilde{S}^{2}}(\ssetb)$ and symmetric spectra
$\sspt_{\tilde{S}^{2}}(\ssetb)$ are equipped with stable model structures using \cite{Hovey:Spt} and the resulting homotopy categories are equivalent (as closed symmetric monoidal triangulated categories) to the usual stable homotopy category of $S^{1}$-spectra. 

If $E$ is a $\P^1$-spectrum, define the $\tilde{S^2}$-spectrum $\RRe_{\C}E$ by $(\RRe_{\C}E)_{i} = \RRe_{\C}E_i$ with structure maps $\RRe_{\C}E_i\wedge \RRe_{\C} (\P^1) = \RRe_{\C}(E_{i}\wedge \P^1) \to \RRe_{\C}E_{i+1}$. The functor $\msing_{\C}$ extends as well to a functor on $\tilde{S^2}$-spectrum. We obtain an adjoint pair of functors
 $\RRe_{\C}:\sspt_T(\sch/\C) \rightleftarrows \sspt_{\tilde{S}^{2}}(\ssetb):\msing_{\C}$
and similarly for ordinary spectra.

\begin{theorem}
The adjoint pairs
 $\RRe_{\C}:\sspt_T(\sch/\C) \rightleftarrows \sspt_{\tilde{S}^{2}}(\ssetb):\msing_{\C}$
and $\RRe_{\C}:\spt_T(\sch/\C) \rightleftarrows \spt_{\tilde{S}^{2}}(\ssetb):\msing_{\C}$
are Quillen adjoint pairs. In the first case, $\RRe_{\C}$ is strict symmetric monoidal.
\end{theorem}
\begin{proof}
 One may argue as in \cite[Theorem A.45]{PPR:KGL}.
\end{proof}

\subsection{Friedlander-Walker's construction}
 
Let $F:\Sch/\C^{op}\to \sSet$ be a presheaf of simplicial sets. 
Let $T$ be a topological space. Define $F(X\times T)$ to by the left Kan extension along $\Sch/\C\to \Top$. That is, $F(X\times T)  = \colim_{T\to Y(\C)}F(X\times Y)$
where the indexing category is the filtered category whose objects are continuous maps 
$T\to Y(\C)$ and whose morphisms are maps of schemes making the obvious triangle commute. 
Now define
$$
F^{\sst}(X) = \diag(d\mapsto F(X\times\Delta^{d}_{top})).
$$
If $F$ is based, then so is $F^{\sst}$. 
If $F=X_+$ is representable, then it is not hard to see that $F(\Delta^{\bullet}_{top}) = \sing X(\C)_+$. As $(-)^{\sst}$ commutes with colimits,  $F^{\sst}$ may also be described as the presheaf
$$
F^{\sst}(X)= \RRe_{\C}(\ihom(X,F)).
$$


Let $E$ be a $\P^1$-spectrum or symmetric $\P^1$-spectrum on $\sch/\C$. Define 
\begin{equation}\label{eqn:fundel}
F(\Delta^{d}_{top},E) = \colim_{\Delta^{d}_{top}\to W(\C)} F(W,E)
\end{equation}
where $F(-,E)$ is the function spectrum. 
We have a simplicial object $F(\Delta^{\bullet}_{top}, E)$ in $\Spt_{\P^1}^{\Sigma}(\Sch/\C)$ and we define
$$
E^{\sst} = |F(\Delta^{\bullet}_{top},E)|.
$$
Note that $E^{\sst}$ is equivalently  described as
$E^{\sst}= (E_{0}^{\sst}, E_{1}^{\sst}, \ldots)$ with structure maps given by
$ E_{i}^{\sst}\wedge \P^1 \to  E_{i}^{\sst}\wedge (\P^1)^{\sst} = (E_{i}\wedge \P^1)^{\sst} \to E_{i+1}^{\sst}$.

Fix a lax symmetric monoidal fibrant replacement functor $(-)^{\fib}$ on $\Spt^{\Sigma}_{\P^{1}}(\Sch/\C)$. 
\begin{definition}\label{def:Qsst}
For a motivic spectrum $E$ on $\sch/\C$, define $\Qsst E = (E^{\fib})^{\sst}$.
If $E$ is a motivic spectrum defined on $\sm/\C$ then define $\Qsst E = \Qsst \L i^*E$.
\end{definition}

 The \textit{semi-topological $E$-theory} is the cohomology theory represented by $\Qsst E$ in $\SH(\C)$, that is $E^{p,q}_{\sst}(X) = [ X_{+}, \Sigma^{p,q}\Qsst E]_{\SH(\C)}$.

Note also that there is a natural convergent spectral sequence 
\begin{equation}\label{eqn:sstss}
E^{2}_{p,q} = H_{p}\big(d\mapsto \ppsh_{q,t}A(X\times\Delta^d_{top}))\big) \Longrightarrow \ppsh_{p+q,t}\Qsst A(X)
\end{equation}
obtained from Theorem \ref{thm:sspseq}.

%

Friedlander-Walker's recognition principle is a useful tool for studying semi-topological cohomology theories. Recall that Voevodsky's $h$-topology on $\Sch/\C$ is the Grothendieck topology whose covers finite collections of maps $\{U_{i}\to X\}$ such that $\coprod U_{i}\to X$ is a universal quotient map.

\begin{theorem}[{\cite[Theorem 2.6]{FW:ratisos}}]\label{thm:rec}
 Let $F\to G$ be a natural transformation of presheaves of abelian groups on $\Sch/{\C}$ such that $F_{h} \to G_{h}$ is an isomorphism of $h$-sheaves. Then 
$ F(\Delta^{\bullet}_{top}) \to G(\Delta^{\bullet}_{top})$
 is a homotopy equivalence of simplicial abelian groups.
\end{theorem}

\begin{corollary}\label{cor:smooth}
 Suppose that $E_1\to E_2$ is a a map of motivic spectra on $\Sch/k$
such that $(\ppsh_{s,t}E_{1}(X\times -))_{h}\to (\ppsh_{s,t}E_{2}(X\times - ))_{h}$  is an
isomorphism of $h$-sheaves for all $s$, $t$. Then $\ppsh_{*,*}\Qsst E_{1}(X)\to \ppsh_{*,*}\Qsst E_{2}(X)$ is an isomorphism.
\end{corollary}
\begin{proof}
 Consider the  spectral sequences (\ref{eqn:sstss}) for $E_{1}$ and $E_{2}$. 
By Theorem \ref{thm:rec} the natural comparison map between the two spectral sequences induces an isomorphism on the $E_{2}$-page. It now follows 
that $\ppsh_{p+q,t}E^{\sst}_1(X)\to \ppsh_{p+q,t}E^{\sst}_2(X)$ is an isomorphism.
\end{proof}

\begin{theorem}\label{thm:sstdef}
\begin{enumerate}
\item The functor 
$\Qsst:\Spt^{\Sigma}_{\P^1}(\Sch/\C) \to \Spt_{\P^1}^{\Sigma}(\Sch/\C)$ 
induces a functor $\Qsst:\SH(\C) \to \SH(\C)$ which is lax symmetric monoidal, coproduct preserving, and triangulated.
\item There is a monoidal natural transformation $\id\to \Qsst$ which fits into a sequence of monoidal natural transformations of lax symmetric monoidal functors
$$
\id \to \Qsst \to \R\msing_{\C}\L\RRe_{\C},
$$
where the composite $\id \to \R\msing_{\C}\L\RRe_{\C}$ is the unit of the adjunction.
\item Let $A$ be a finite abelian group and $MA$ a Moore spectrum for $A$. The natural transformation $\id\wedge MA \to \Qsst\wedge MA$ is an equivalence.
\end{enumerate}
\end{theorem}
\begin{proof}
\textit{(1)}
If $E\to F$ is a motivic stable equivalence  then  $( E^{\fib})_i \to ( F^{\fib})_i$ is a schemewise equivalence for each $i$. Since filtered colimits preserve weak equivalences, we have that  $\Qsst:\SH(\C)\to \SH(\C)$ is well-defined. 

We have schemewise, levelwise weak equivalences $\coprod X_{i}^{\fib}\wkeq(\coprod X_i)^{\fib}$, which implies that $\Qsst$ preserves coproducts.
Homotopy fiber and cofiber sequences agree in $\sspt_{\P^1}(\Sch/\C)$ and so 
$F(W,E^{\fib}) \to F(W, F^{\fib}) \to F(W,G^{\fib})$ is a homotopy cofiber sequence for all $W$. Filtered colimits preserve homotopy fiber sequences and so for each $d$,
$$
F(\Delta^{d}_{top},E^{\fib}) \to F(\Delta^{d}_{top}, F^{\fib}) \to F(\Delta^{d}_{top}, G^{\fib})
$$
is a homotopy cofiber sequence. Taking the realization yields a homotopy cofiber sequence
 $$
|F(\Delta^{\bullet}_{top}, E^{\fib})| \to |F(\Delta^{\bullet}_{top}, F^{\fib})| \to |F(\Delta^{\bullet}_{top}, G^{\fib})|,
$$
which implies that $\Qsst:\SH(\C)\to \SH(\C)$ is a triangulated functor.

\textit{(2)}
 The symmetric monoidal structure on $\SH(\C)$ is defined by  $(-)^{\cof}\wedge (-)^{\cof}$ in $\Spt_{\P^1}^{\Sigma}(\Sch/\C)$, where $(-)^{\cof}$ is a cofibrant replacement functor. We have natural maps $(\Qsst F)^{\cof}\wedge (\Qsst G)^{\cof} \to \Qsst F \wedge \Qsst G \to \Qsst (F\wedge G)$, 
which implies that  $\Qsst:\SH(\C)\to \SH(\C)$ is lax  symmetric monoidal.

If $E$ is a symmetric $\P^1$-spectrum, we have a symmetric $\tilde{S}^{2}$-spectrum $E^{\sst}(\C)$ defined by $E^{\sst}(\C)=(E_{0}^{\sst}(\C), E_{1}^{\sst}(\C), \ldots)$ with structure maps 
given as the composite
$E_{i}^{\sst}(\C)\wedge \tilde{S}^{2} = (E_{i}\wedge \P^1)^{\sst}(\C) \to E_{i+1}^{\sst}(\C)$.  
We claim that the spectrum $\Qsst E(\C):= (E^{\fib})^{\sst}(\C)$ agrees with $\L\RRe_{\C}E$ in $\SH$.  As both functors $\L\RRe_{\C}(-)$ and $\Qsst(-)(\C)$ preserve stable equivalences, it suffices to verify the claim when $E$ is both cofibrant and fibrant. If $E$ is cofibrant then $\L\RRe_{\C}E \wkeq E^{sst}(\C)$. If $E$ is fibrant then each $E^{\sst}_{i} \to (\Qsst E)_{i}$ is a schemewise equivalence, which implies that $E^{\sst}(\C) \wkeq  \Qsst E(\C)$. 
If $F$ is a motivic space, the argument of \cite[Lemma 3.2]{FHW:sst} shows that the natural transformation $F\to \msing \RRe F$ factors through a natural transformation $F\to F^{sst}$. If $E$ is a $\P^1$-spectrum this natural transformation is compatible with the structure maps  and so the natural transformation $E\to \msing \Qsst E(\C)$ factors through $E\to E^{\sst}$. These are evidently lax symmetric monoidal transformations.

\textit{(3)}
Let $E$ a motivic spectrum and write $E'=E\wedge MA$. Consider a smooth complex variety  $X$  and
let $s$ and $t$ be integers. Write $F_{s,t}(U) =\ppsh_{s,t}E'(X)$ for the constant presheaf of abelian groups on $\sch/\C$. We have a map of presheaves $F_{s,t}(-)\to \ppsh_{s,t}E'(X\times -)$. By \cite[Corollary 2.5]{Yagunov:rigid}
this map induces an isomorphism on $\mcal{O}^{h}_{X,x}$ for any closed point $x\in X$ and smooth $X$. This implies that the map $(F_{s,t})_{h}\to (\ppsh_{s,t}E'(X\times -))_{h}$ of $h$-sheaves is an isomorphism for all $s,t$. It follows from Corollary \ref{cor:smooth} that $\ppsh_{s,t}E'(X) = \ppsh_{s,t}\Qsst E'(X)$ for all $s,t$.

\end{proof}

\section{Examples and basic properties}\label{sec:ex}
The two original examples of a semi-topological cohomology theory are morphic cohomology introduced by Friedlander-Lawson \cite{FL:algco} and semi-topological $K$-theory introduced by Friedlander-Walker \cite{FW:sstfct}. In this section, we verify $\Qsst$ as defined above recovers these theories. In addition, we show that under certain circumstances the group $E^{2q,q}_{sst}(X)$ agrees with $E^{2q,q}(X)/\sim_{alg}$, the algebraic theory modulo algebraic equivalence. We also show that for effective motivic spectra, the coefficients of the associated semi-topological cohomology theory agrees with that of the associated topological theory. 

\subsection{Almost fibrant semi-topological spectra}
Relating the motivic version of Friedlander-Walker's construction given in the previous section to their original construction boils down to recognizing when $\Qsst$ produces an $\Omega_{\P^1}$-spectrum. 
This is a consequence of the following simple variation of \cite[Corollary 2.7]{FHW:sst}. 
 
\begin{theorem}[c.~f.~{\cite[Corollary 2.7]{FHW:sst}}]\label{thm:27}
 Let 
$F_3\to F_2\to F_1$
be a sequence of natural transformations of simplicial presheaves on $\Sch/\C$ such that each $F_{i}$ is a presheaf of homotopy commutative group-like $H$-spaces and the induced maps on $\pi_{0}$ are group homomorphisms. Suppose further that $F_{1}(U)$, $F_{2}(U)$ are connected for all smooth $U$  and $F_3(X)\to F_2(X)\to F_1(X)$ is constant for all $X$ in $\Sch/\C$. If it is a homotopy fiber sequence when $X$ is smooth, then 
$$
F_3^{\sst}(U) \to F_2^{\sst}(U) \to F_1^{\sst}(U)
$$
is also homotopy fiber sequence for all smooth $U$.
\end{theorem}
\begin{proof}
 A similar argument  as that used in  loc.~cit.~applies here. Consider the homotopy fiber sequence  $F_{3}'\to F_{2}'\to F_1'$ of presheaves on $\Sch/\C$,
where $F_{2}'(X)$ and $F_{1}'(X)$ are defined to be the connected components of $F_{2}(X)$ and $F_{1}(X)$ and define $F_{3}'(X)= \hofib(F_{2}'(X) \to F_{1}'(X))$. Filtered colimits preserve homotopy fiber sequences and so for $U$ smooth, $F_{3}'(U\times \Delta^{d}_{top})\to F_{2}'(U\times\Delta^{d}_{top})\to F_1'(U\times\Delta^{d}_{top})$ is a homotopy fiber sequence of homotopy commutative group-like $H$-spaces. As  $F_1'(U\times\Delta^{d}_{top})$ and $F_{2}'(U\times \Delta^{d}_{top})$ are connected, \cite[Theorem B.4]{BF} implies that $(F_{3}')^{\sst}(U)\to (F_{2}')^{\sst}(U)\to (F_1')^{\sst}(U)$ is a homotopy fiber sequence. 
For smooth $X$ we have that $F_{i}'(U\times X)\to F_{i}(U\times X)$ for each $i$. Therefore we have that 
 $(F_{i}')^{\sst}(U)\to F_{i}^{\sst}(U)$ is a weak equivalence for each $i$ by \cite[Theorem 2.6]{FHW:sst}, yielding the result. 
\end{proof}

\begin{corollary}\label{cor:fib}
Let $E$ be a presheaf of homotopy commutative group-like $H$-spaces on $\Sch/\C$. Suppose that there is a presheaf of homotopy commutative group-like $H$-spaces $E_1$ and a map $E\to \Omega E_1$ inducing weak equivalences $E(X)\wkeq \Omega E_{1}(X)$ for all smooth $X$, and that $E_1(X)$ is connected for all smooth $X$. 
\begin{enumerate}
 \item If $E$ satisfies Nisnevich (resp.~$cdh$) descent on $\Sm/\C$  then so does $E^{\sst}$.
\item If $E$ is homotopy invariant on $\Sm/\C$ then so is $E^{\sst}$.
\item If $E$ satisfies Nisnevich descent on $\Sm/\C$ and is homotopy invariant 
then $\Omega_{\P^1}^r E^{\sst} \to (\Omega_{\P^1}^r E)^{\sst}$ is a schemewise equivalence on $\Sm/\C$. 
\end{enumerate}
\end{corollary}
\begin{proof}
First observe that $(\Omega E_1)^{\sst}(X) \wkeq \Omega E_1^{\sst}(X)$ for any smooth $X$ because $(\Omega E_1)^{sst}(X) \to * \to E_1^{sst}(X)$ is a homotopy fiber sequence by Theorem \ref{thm:27}. If $E$ satisfies Nisnevich (resp.~$cdh$) descent, then so does $E_1$. 
Consider the square
$$
\xymatrix{
E_1^{\sst}(Y) \ar[d]\ar[r] & E_1^{\sst}(B) \ar[d]\\
E_1^{\sst}(X) \ar[r] & E_1^{\sst}(A)
}
$$
associated to a distinguished square on $\Sm/\C$.
Write $F(U) = \hofib(E_1(Y\times U)\to E_1(B\times U)$ and $G(U) = \hofib(E_1(X\times U) \to E_1(A\times U)$. By Theorem \ref{thm:27}, $F^{\sst}(\C)$ is the homotopy fiber of the top horizontal arrow of this square and $G^{\sst}(\C)$ is the homotopy fiber of the bottom horizontal arrow. Since $F(U)\to G(U)$ is a weak equivalence for all smooth $U$, it follows from \cite[Theorem 2.6]{FHW:sst} that $F^{\sst}(\C)\to G^{\sst}(\C)$ is a homotopy equivalence and therefore $E_1^{\sst}$ satisfies descent on $\sm/\C$ and thus so does $E^{\sst}$.

Similarly, since $E$ is homotopy invariant on $\sm/\C$, it follows from 
\cite[Theorem 2.6]{FHW:sst} that $E^{\sst}$ is homotopy invariant on $\sm/\C$.

If $E^{\sst}$ satisfies descent and is homotopy invariant then for any smooth $X$ we have a homotopy fiber sequence $\Omega_{\P^1}E_1^{\sst}(X) \to E_1^{\sst}(\P^{1}\times X) \to E_1^{\sst}(\P^{0} \times X)$.
Theorem \ref{thm:27} implies that for any smooth $X$ we also have a homotopy fiber sequence
$(\Omega_{\P^1}^{1}E_1)^{\sst}(X) \to E_1^{\sst}(\P^{1}\times X) \to E_1^{\sst}(\P^{0} \times X)$, which immediately implies that we have $(\Omega_{\P^1}^{1}E_1)^{\sst}(X)\wkeq \Omega_{\P^1}^{1}E_1^{\sst}(X)$. Note that the inclusion $\P^{0}\to \P^{1}$ induces a surjection $\pi_{k}E_1(\P^{1}\times X)\to \pi_kE_1(\P^0\times X)$  as it has a splitting induced by $\P^{1}\to \P^0$. This implies that $\pi_{0}\Omega_{\P^1}E_1(X)\to \pi_0E_1(\P^1\times X)$ is injective and in particular that $\Omega_{\P^1}^{1}E_1(X)$ is connected when $X$ is smooth. Iterating the previous argument, we obtain $(\Omega_{\P^{1}}^{r}E_{1})^{\sst}(X)\wkeq \Omega_{\P^{1}}^{r}E_{1}^{sst}(X)$. Finally, since these are connected on $\Sm/\C$ we have that $\Omega(\Omega_{\P^1}^{r}E_1)^{\sst}(X)\wkeq (\Omega\Omega_{\P^1}^{r}E_1)^{\sst}(X)$. The final statement now follows by combining these weak equivalences.
\end{proof}

\begin{definition}\label{def:almostfib}
We say that a motivic spectrum $E$ on $\sch/\C$ is \textit{almost fibrant} on $\sm/\C$ if each restriction $i_*E_{k}$ satisfies $cdh$-descent on $\sm/\C$, is homotopy invariant, and $E_{k}(X)\to\Omega_{\P^1}E_{k+1}(X)$ is a weak equivalence for all smooth $X$. 
\end{definition}
If $E$ is almost fibrant on $\sm/\C$, then $i_{*}E_k(X) \to (\R i_* E)_{k}(X)$ is a weak equivalence for any smooth $X$ and $k$.

\begin{corollary}\label{cor:omega}
 Let $E$ be a motivic spectrum such that $\ppsh_{i,-q}E(X) = 0$ for $i<-2q$, $q\geq 0$, and smooth $X$. Then $\Qsst E$ is almost fibrant on $\sm/\C$. In particular if $E$ is itself almost fibrant on $\sm/\C$ then $(\Qsst E)_{i}(X)\wkeq E_{i}^{\sst}(X)$ and we have
$$
E_{sst}^{s,q}(X) := [\Sigma^{-s,-q}X_+, \Qsst E]_{\SH(\C)} = \pi_{2q-s}E^{\sst}_{q}(X)
$$ 
for all smooth $X$, any $q$,  and $s\leq 2q$ (where $E^{\sst}_{q} = \Omega^{q}_{\P^{1}}E^{\sst}_{0}$, if $q\leq 0$).
\end{corollary}
\begin{proof}
We may replace $E$ by $E^{\fib}$. Each $E_{q}$ is a presheaf of homotopy commutative $H$-spaces, for $X$ smooth we have  $E_{q}(X)\wkeq \Omega\Omega_{\G_{m}}E_{q+1}(X)$ and $\Omega_{\G_{m}}E_{q+1}(X)$ is connected under the hypothesis on $E$. It then follows from Corollary \ref{cor:fib} that $\Qsst E$ is almost fibrant on $\sm/\C$. 
%
\end{proof}

\subsection{Morphic Cohomology}\label{sec:morphic}
We begin with morphic cohomology, introduced by Friedlander-Lawson  \cite{FL:algco}. The main point concerning representability is the reformulation due to
Friedlander-Walker \cite{FW:ratisos} that morphic cohomology for smooth quasi-projective varieties can be obtained by applying  $(-)^{\sst}$ to the complex of equidimensional cycles. 

Let $X$ be a complex variety. Recall the presheaf $z_{equi}(X,0)(-)$ of equidimensional cycles  constructed in \cite{SV:rel}. This is the unique $qfh$-sheaf on $\Sch/\C$ such that for a normal variety $U$ the group $z_{equi}(X,r)(U)$ is the free abelian group generated by closed, irreducible subvarieties $V\subseteq U\times X$ which are equidimensional of relative dimension $0$ over some irreducible component of $U$. 
When $Y$ is projective, then $z_{equi}(Y,0)(-) = \Z_{tr}(Y)$ is the free presheaf with transfers generated by $Y$.

For a presheaf of groups $F$, we write  $C_*F = F(-\times\Delta^{*}_{\C})$.
For any smooth $X$ we have
$$
\pi_{2q-p}C_*z_{equi}(\A^{q},0)(X) = H^{p}_{\mcal{M}}(X,\Z(q)),
$$ 
where the right hand side is the Suslin-Voevodsky definition of motivic cohomology, see e.g. \cite[Corollary 18.4, Theorem 19.1]{MVW}.

\begin{definition}
Let $Y$ be a complex variety, define the motivic spectrum $\MZ_{c}(Y)$ (on $\sch/\C$) by $\MZ_{c}(Y)_{k} = C_*z_{equi}(Y\times \A^{k},0)$ and structure maps are given by
\begin{align*}
C_* & z_{equi}(Y\times\A^{k},0) \wedge \P^1\to C_*z_{equi}(Y\times\A^{k},0) \wedge z_{equi}(\P^1,0) \\ & \to C_*z_{equi}(Y\times\A^{k},0) \wedge z_{equi}(\A^1,0)
\to   C_*z_{equi}(Y\times\A^{k+1},0).
\end{align*}
\end{definition}
Write $\MZ = \MZ_{c}(\C)$ which is an almost fibrant model for the motivic cohomology spectrum. More generally if $Y$ is projective then $\MZ_{c}(Y)$ is an almost fibrant model (see Definition \ref{def:almostfib}) for $Y_+\wedge \MZ$ on $\sm/\C$. If $Y$ is quasi-projective, then it is an almost fibrant model on $\sm/\C$ for $(\overline{Y}/\overline{Y}-Y_{\infty})\wedge \MZ$, where $Y\subseteq \overline{Y}$ is a projectivization and $Y_{\infty}=\overline{Y}-Y$.

By \cite[Corollary 3.5]{FW:ratisos}, if $X$ is smooth and quasi-projective then $(\MZ_{k})^{\sst}(X)$ computes morphic cohomology, in the sense that $\pi_{j}(\MZ_{k})^{\sst}(X) = L^{k}H^{2j-k}(X)$. By Corollary \ref{cor:omega} it then follows that $\MZ^{\sst}:=\Qsst\MZ$ represents morphic cohomology of smooth complex varieties.
The pairings 
$z_{equi}(\A^{n},0)\wedge z_{equi}(\A^{m},0) \to z_{equi}(\A^{n+m},0)$ give $\MZ$ the structure of a commutative motivic ring spectrum which induces the usual product structure on motivic cohomology, see e.g. \cite{Weibel:prod}. This gives $\MZ^{\sst}$ the structure of a commutative strict ring spectrum as well. In summary we have verified the following proposition. 
\begin{proposition}
 The commutative motivic strict ring spectrum  $\MZ^{\sst}$  represents morphic cohomology theory in the sense that
$$
[X_+, \Sigma^{p,q}\MZ^{\sst}]_{\SH(\C)} = L^qH^{p}(X)
$$ 
for all $p$, $q$, and smooth quasi-projective complex varieties $X$ . 
\end{proposition}

\subsection{Semi-topological {$K$}-theory} 
Consider $K^{W}(X)$, the Waldhausen $K$-theory space  of bounded complexes of big vector bundles on $X$. This defines a presheaf of simplicial sets on $\Sch/\C$.
The \textit{semi-topological $K$-theory} of $X$ is defined by
$$
K_{q}^{\sst}(X) = (K^{W})^{\sst}(X).
$$

Voevodsky \cite{Voev:ICM} constructed a $\P^{1}$-spectrum representing algebraic $K$-theory on smooth schemes which is essentially unique, we refer to \cite{PPR:KGL} for details. 
Write $\KGL$ for the $\P^1$-spectrum on $\Sch/\C$  defined as follows.  Let $K^{W}\to \mathbb{K}^{W}$ be a motivic fibrant replacement (in $\MMb(\Sch/\C)$). Since $K$-theory is homotopy invariant and satisfies Nisnevich and $cdh$-descent on $\sm/\C$, we have that
 $K^{W}(X) \to \mathbb{K}^{W}(X)$ is a weak equivalence for any smooth $X$. Now let $\mathbb{K}^{W}\to \mcal{K}$ be a cofibrant replacement, so that $\mcal{K}$ is a motivic fibrant and cofibrant model for $K^{W}$ on $\sch/\C$.


 Define $\epsilon:\mcal{K}\wedge \P^1\to \mcal{K}$ to be a lift of the map $K^{W}\wedge \P^{1} \to K^{W}$ in $\hb(k)_{cdh}$ given by  multiplication with $h =[\mcal{O}(-1)]-[\mcal{O}] \in K_{0}(\P^{1})$. The resulting $\P^1$-spectrum $\rKGL:= (\mcal{K}, \mcal{K},\ldots )$ is almost fibrant on $\Sm/\C$ and its restriction to $\Sm/\C$ agrees with the usual construction of a spectrum representing algebraic $K$-theory. 
Recall \cite{SO:bott, GepnerSnaith} that there is an element 
$\beta\in \pi_{2,1}(\Sigma^{\infty}\BGmp)$ such that we have a stable equivalence $i_*\rKGL\wkeq (\Sigma^{\infty}\BGmp)[\beta^{-1}]$. We write $\beta$ as well for its image in 
$\pi_{2,1}(\Qsst\Sigma^{\infty}\BGmp)$. Recall also that in \cite{RSO:strict} a strict motivic ring spectrum $\KGL$ on $\sm/\C$ is constructed.

\begin{proposition}
 The commutative strict motivic ring spectrum  $\KGL^{\sst}:=\Qsst\KGL$ represents semi-topological $K$-theory in the sense that 
$$
[\Sigma^{p,q}X_+, \QQ^{\sst}\KGL]_{\SH(\C)} = K^{\sst}_{p}(X)
$$ 
for all $p$, $q$, and smooth quasi-projective complex varieties $X$. Moreover, there is a stable equivalence $\KGL^{\sst}\wkeq (\Qsst \Sigma^{\infty}\BGmp)[\beta^{-1}]$.
\end{proposition}
\begin{proof}
It follows immediately from Corollary \ref{cor:omega} and the definition of semi-topological $K$-theory that $\Qsst\rKGL$ represents semi-topological $K$-theory. In \cite{RSO:strict} it is shown that there is a stable equivalence of $\P^1$-spectra $\rKGL\to \KGL$ and so $\KGL^{\sst}$ is a strict motivic ring spectrum representing semi-topological $K$-theory.
The second statement follows immediately in light of the fact that $\Qsst$ and $\hocolim_{\N}$ commute.
\end{proof}

\subsection{Semi-topological cobordism}\label{sub:cob}
Semi-topological cobordism was originally defined in \cite{thesis} using a motivic version of Friedlander-Walker's construction (defined in terms of $S^{1}$-$T$-bispectra) which did not necessarily preserve strict ring spectra. 
Recall that the motivic cobordism spectrum  $\MGL$ is defined by
$$
\MGL_{n} = \colim_{m\geq n} \Th(\gamma_{n,mn}),
$$
where $\gamma_{n,mn}$ 
is the tautological bundle $\gamma_{n,mn}\to \Grass_{n,mn}$ over the Grassmannian of $n$ planes in $\A^{mn}$ and $\Th(-)$ is the Thom space. It is a commutative motivic symmetric ring $T$-spectrum, where $T=\A^{1}/\A^{1}-0$,  and we write $\MGL$ for the associated $\P^1$-spectrum. We refer to \cite[Section 2.1]{PPR:MGL2KGL} for full details. The \textit{semi-topological cobordism} spectrum is $\MGL^{\sst} := \QQ^{\sst}\MGL$.



\begin{remark}
The usual construction of $\MGL$ can be repeated on $\sch/\C$. 
Temporarily write $\MGL'$ for the resulting motivic cobordism spectrum on $\sch/\C$. Then $i^*\MGL = \MGL'$. 
The functor $i^{*}:\sspt_{T}(\Sm/k) \to \sspt_{T}(\Sch/k)$ preserves level equivalences between level cofibrant objects.
As $\MGL$ is level cofibrant and the map $\MGL^{\cof}\to \MGL$ is a level equivalence (where $(-)^{\cof}$ is a stable cofibrant replacement),  it follows that $\L i^{*} \MGL \to i^{*}\MGL=\MGL'$ is a level equivalence. 
\end{remark}

As $\MGL$ is a strict commutative ring spectrum in $\sspt_{\P^1}(\Sch/\C)$,  $\MGL^{\sst}$ is also a strict commutative ring spectrum. The lax monoidal natural transformation $id\to \QQ^{\sst}$ implies that we have a canonical ring map
$\MGL \to \MGL^{\sst}$,
and so $\MGL^{\sst}$ has a canonical orientation.

In Section \ref{sec:bott} we use the \textit{periodic semi-topological cobordism spectrum} 
$P\MGL^{\sst} = \bigvee_{n\in \Z}\Sigma^{2n,n}\MGL^{\sst}$.  
This is a ring spectrum in the evident way, see \cite{GepnerSnaith} and has a strict ring model, see \cite[Proposition 5.4]{GepnerSnaith} and \cite[Remark 3.7]{RSO:strict}

\begin{proposition}
 There is an element $\beta\in \pi_{2,1}\BGL$ such that there is natural equivalence $P\MGL^{\sst}\wkeq (\Qsst\Sigma^{\infty}\BGL_{+})[\beta^{-1}]$.
\end{proposition}
\begin{proof}
 This follows from \cite[Corollary 3.10]{GepnerSnaith}, that $\Qsst$ commutes with coproducts, and that $\Qsst$ and $\hocolim_{\N}$ commute.
\end{proof}

\subsection{Algebraic equivalence}
A fundamental property of morphic cohomology and semi-topological $K$-theory is that $L^{q}H^{2q}(X)$ and $K_{0}^{sst}(X)$  are computable respectively in terms of the group of codimension $q$-cycles modulo algebraic equivalence and   the Grothendieck group of vector bundles modulo algebraic equivalence. This relationship can be generalized to a wide class of motivic spectra.

\begin{definition}\label{def:algeq}
Let $A$ be a presheaf of sets or groups on $\Sch/\C$. 
Two elements  $\alpha, \beta\in A(X)$ are 
said to be \textit{algebraically equivalent} if
there is a smooth, connected curve $C$, two closed points $c_{1}$, $c_{2}\in C$  and an element $\gamma\in A(X\times C)$ such that $\gamma|_{X\times c_{1}} = \alpha$ and $\gamma|_{X\times c_{2}} = \beta$. Write $\sim_{alg}$ for the equivalence relation that this generates.
\end{definition}

\begin{lemma}\label{lem:p0algeq}
 Let $E:\Sch/\C^{op}\to \sset$ be a presheaf of Kan complexes. Then $\pi_{0}E(X\times\Delta^{\bullet}_{top}) = \pi_{0}E(X)/\sim_{alg}$.
\end{lemma}
\begin{proof}
 We have a coequalizer diagram 
$$
\pi_{0}E(X\times\Delta^{1}_{top})\rightrightarrows\pi_{0} E(X) \to \pi_{0}E^{\sst}(X),
$$
where the left hand arrows are induced by the respective inclusions $\Delta^{0}_{top}\to \Delta^{1}_{top}$ at $0$ and $1$. Replacing $E$ by $\ihom(X,E)$, it suffices to treat the case $X=\spec(\C)$.

If $C$ is connected, any two points $c_{0},c_{1}\in C(\C)$  can be joined by a continuous map $f:\Delta^{1}_{top}\to C(\C)$ such that $f(i) = c_{i}$.  
It follows that we have a well-defined surjection $\pi_{0}E(\C)/\sim_{alg} \to \pi_{0}E^{\sst}(\C)$. 

For injectivity, 
suppose that $\beta$, $\alpha \in \pi_{0}E(\C)$ map to the same element of  $\pi_{0}E^{\sst}(\C)$. This means that there is a $Y\in \Sch/\C$, a continuous map $g:\Delta^{1}_{top}\to Y(\C)$ and a $\gamma\in F(Y)$ such that $\alpha= \gamma|_{g(0)}$ and $\beta = \gamma|_{g(1)}$. 
The required injectivity follows easily from \cite[p. 56]{Mumford}, which asserts that if $W$ is irreducible and quasi-projective and
$w_0$, $w_1\in W(\C)$ lie in the same topological component, then there exists  
a smooth connected curve joining these points. 
\end{proof}

\begin{proposition}\label{prop:algeq}
 Suppose that $E^{k,q}(X) = 0$ for all smooth $X$ and $k>2q$. Then 
$E^{2q,q}_{\sst}(X) = E^{2q,q}(X)/\sim_{alg}$.
\end{proposition}
\begin{proof}
 By Corollary \ref{cor:omega} we have that $E^{2q,q}_{\sst}(X) = \pi_{0}E_{q}^{sst}(X)$ and so the result follows from Lemma \ref{lem:p0algeq}.
\end{proof}

\begin{corollary}\label{cor:mglalgeq}
Let $X$ be smooth, then for any $q$ there are natural isomorphisms 
$$
\Omega^{q}(X)/\sim_{alg} \xrightarrow{\iso}\MGL^{2q,q}(X)/\sim_{alg} \xrightarrow{\iso}  \MGL_{\sst}^{2q,q}(X)  
$$
 where $\Omega^*(-)$ is Levine-Morel algebraic cobordism.
\end{corollary}
\begin{proof}
Since $\MGL^{p,q}(X) = 0$ for $p\geq 2q$ (which follows for example from the slice spectral sequence, see Section \ref{sec:bott}), the previous result implies the second isomorphism. Levine \cite[Theorem 3.1]{Levine:cob} shows that $\Omega^{q}(X)\to \MGL^{2q,q}(X)$ is a natural isomorphism for any smooth $X$, which implies the first isomorphism.
\end{proof}

\subsection{Coefficients of semi-topological theories}\label{sub:coeff}
Now we turn to a generalization of the relationships $L^{q}H^{p}(\C)=H^{p}(\pt)$ and $K^{\sst}_0(\pt)=\ku^0(\pt)$. 

We begin by recalling Voevodsky's slice tower.
Let $\SH^{eff}(\C)\subseteq \SH(\C)$ be the smallest localizing subcategory containing all $X_{+}$, $X\in \Sm/\C$ and let $\Sigma^{q}_{\P^1}\SH^{eff}(\C)\subseteq \SH(\C)$ be the smallest localizing subcategory containing $\Sigma^{q}_{\P^1}E$ for $E$ in $\SH^{eff}(\C)$. The inclusion $i_{q}:\Sigma^{q}_{\P^1}\SH^{eff}(\C)\subseteq \SH(\C)$ has a right adjoint, $r_{q}$. Define $f_{q} = i_{q}r_{q}$. There is a natural map $f_{q} E \to E$ which is the universal for maps from objects in $\Sigma^{q}_{\P^1}\SH^{eff}(\C)$ to $E$. Define $s_{q}E = \cofiber(f_{q+1}E\to f_{q}E)$ which is the \textit{$q$th slice} of $E$.

Using Voevodsky's computation \cite{Voev:slice} that $s_{0}\SS = \MZ$, Pelaez \cite{Pelaez} shows that the slices $s_{q}E$ are all $\MZ$-modules. {\O}stv{\ae}r-R{\"o}ndigs \cite{RO:MZ} show that Voevodsky's big category of motives $\DM(\C)$ is equivalent, as a monoidal, triangulated category, to the homotopy category of $\MZ$-modules. Write $\DM^{eff}(\C)$ for the smallest localizing subcategory of $\DM(\C)$ containing all $X_{+}\wedge \MZ$, $X\in \Sm/\C$.

\begin{lemma}\label{lem:motcon}
 For any $E$ in $\DM^{eff}(\C)$ we have isomorphisms
$$
\ppsh_{*,0}(\Qsst E)(\C) \xrightarrow{\iso} \pi_{*}\L\Re_{\C}E.
$$
\end{lemma}
\begin{proof}
 Since $\L\Re_{\C}$ and 
$\ppsh_{*,0}$ are compatible with triangles and coproducts, it suffices to verify the lemma in the case 
$E = X_+\wedge \MZ$, where $X$ is a smooth projective complex variety. We have the almost fibrant model $\MZ(X):=(C_*z_{equi}(X\times \A^{k},0))_{k\geq 0}$ for $X_{+}\wedge \MZ$ (see Section \ref{sec:morphic}).
Therefore $\ppsh_{p,0}(\Qsst E)(\C) = \pi_{p}C_*z_{equi}(X,0)(\Delta^{\bullet}_{top})$ if $p>0$.

For projective $Y$, define $\mcal{Z}_{0}(Y) = (\coprod_{n\geq 0}\mathrm{Sym}_{n}Y(\C))^{+}$ where $(-)^+$ denotes the group completion of the displayed topological monoid, equipped with a topology via the quotient topology. 
For any quasi-projective complex variety $W$ define $\mcal{Z}_{0}(W) = \mcal{Z}_{0}(\overline{W})/\mcal{Z}_{0}(\overline{W}-W)$, where $W\subseteq \overline{W}$ is a projectivization. 
We have a natural homotopy equivalence $\sing\mcal{Z}_{0}(W)\wkeq \MZ_{c}^{\sst}(W)(\C)$. Indeed, by \cite[Proposition 3.1]{FW:ratisos}, there is a natural homotopy equivalence 
$z_{equi}(W,0)(\Delta^{\bullet}_{top})\wkeq \sing\mcal{Z}_{0}(W)$ and the natural map $z_{equi}(W,0)(\Delta^{\bullet}_{top})\to C_*z_{equi}(W,0)(\Delta^{\bullet}_{top})=\MZ_{c}^{\sst}(W)(\C)$ is a homotopy equivalence, see e.g. \cite[Lemma 1.2]{FW:compK}. 

By the Dold-Thom theorem, $\pi_{n}\mcal{Z}_{0}(W)= H^{BM}_{n}(W(\C),\Z)$, where $H^{BM}(-)$ is Borel-Moore singular homology. For a projective complex variety $X$, the adjoint of the map
$\mcal{Z}_{0}(X\times\A^{k})\wedge \P^{1}(\C) \to \mcal{Z}_{0}(X\times\A^{k+1})$
induces the suspensions isomorphism $\tilde{H}_{n}(X(\C)_{k}\wedge S^{2k})\iso \tilde{H}_{n+2}(X(\C)_+\wedge S^{2k+2})$ in homology. Applying $\sing$ and using the above identifications, 
$$
z_{equi}(X\times \A^{k+1},0)(\Delta^{\bullet}_{top})  \to \Omega_{\tilde{S^2}}z_{equi}(X\times \A^{k},0)(\Delta^{\bullet}_{top})
$$
is a homotopy equivalence. It follows that $\L\Re_{\C}\MZ(X)$ is an $\Omega_{\tilde{S^2}}$-spectrum. In particular if $p\geq 0$ then $\ppsh_{p,0}(\Qsst \MZ(X))(\C) \xrightarrow{\iso} \pi_{p}\L\Re_{\C}(\MZ(X))$ is an isomorphism and both groups are zero if $p<0$

\end{proof}

\begin{theorem}\label{thm:toppt}
 If $E$ is in $\SH(\C)^{eff}$ and $q\geq 0$, then
$$
E_{\sst}^{-p,q}(\C)= \ppsh_{p,-q}(\Qsst E)(\C) \xrightarrow{\iso} \pi_{p}\L\Re_{\C}E.
$$
\end{theorem}
\begin{proof}
If  $q\geq 0$ then $\pi_{p,-q}E = \pi_{p+2q,0}(\Sigma_{\P^{1}}^{q}E)$. If $E$ is effective, then so is $\Sigma_{\P^{1}}^{q}E$. It thus suffices to verify the theorem  for $q=0$. 
 Since $\L\Re_{\C}$ and 
$\ppsh_{p,0}$ are compatible with triangles and coproducts, it suffices to consider the case $E=X_{+}$ for smooth $X$.  By \cite[Theorem 4, Proposition 6.9]{Levine:conv} the slice tower $\cdots \to f_{2}X_{+}\to f_{1}X_{+} \to f_{0}X_{+}= X_{+}$ converges. 
By Theorem \ref{thm:sstdef}(3), this implies that the tower obtained from applying $\Qsst\wedge MA$ also converges (where $MA$ is the Moore spectrum associated to a finite abelian group $A$). By \cite[Lemma 6.1]{Levine:comp}, if $q\geq \dim(X)+1$ then all of the homotopy sheaves of $f_{q}X_+$ are torsion. This implies that $\Qsst(f_{q}X_+)\wedge M\Q= *$ for $q\geq \dim(X)+1$ and so the tower $\cdots \to \Qsst (f_{2}X_{+})\to \Qsst (f_{1}X_{+}) \to \Qsst(X_{+})$ converges.
The associated topological tower $\cdots\to \L\Re_{\C}(f_{2}X_{+})\to \L\Re_{\C}(f_{1}X_{+})\to \L\Re_{\C}(X_{+})$ also converges by \cite[Theorem 5.2]{Levine:comp}.
We thus have a comparison of convergent spectral sequences
\begin{align*}
E_{p,q}^2=\ppsh_{p+q,0} (\Qsst s_qX_{+})(\C)&\Longrightarrow \ppsh_{p+q,0}(\Qsst X_{+})(\C) \\
&\downarrow \\
E^{2}_{p,q} = \pi_{p+q} \L\Re_{\C}( s_qX_{+}) &\Longrightarrow \pi_{p+q}X_{+}.
\end{align*}
Now if $E$ is in $\SH(\C)^{eff}$ then each $s_{q}E$ is in $\DM^{eff}(\C)$. In particular, Lemma \ref{lem:motcon} implies that the comparison map induces  an isomorphism on the $E^2$-pages of these spectral sequences. It follows that $\ppsh_{p,0}\Qsst E(\C) \xrightarrow{\iso} \pi_{p}\L\Re_{\C}E$ is an isomorphism.

\end{proof}
\begin{remark}
  Note that $\KGL$ is not effective and so Theorem \ref{thm:toppt} does not assert that $\KGL_{\sst}^{p,0}(\C)=\pi_{-p}\L\Re_{\C}(\KGL) = \KU^{p}(\pt)$ (which would be absurd). However the theorem does apply to the $\P^1$-connective $K$-theory, $\kgl:=f_{0}\KGL$. 
One has $\KGL_{\sst}^{p,0}(\C) = \kgl_{\sst}^{p,0}(\C)$ and it can be shown that $\L\RRe_{\C}(\kgl) = \ku$.

The theorem applies to $\MGL$ and so $\MGL_{\sst}^{p,q}(\C) = MU^{p}$ for all $q\geq 0$. However the groups $\MGL_{\sst}^{p,q}(\C)$ for $q\leq 0$ are also interesting. 
See Remark \ref{rem:fullpt} for a complete calculation of these groups.

\end{remark}

\section{Applications to semi-topological cobordism}\label{sec:bott}
In this section we give two applications to semi-topological cobordism, which rely on the multiplicative nature of $\Qsst$. 
The first is a natural isomorphism of the form $(\oplus_{q}\MGL^{2q+*,q}_{\sst}(X))[s^{-1}] \iso \oplus_{q}MU^{2q+*}(X)$,
where $s$ is a lift of the Friedlander-Mazur $s$-element in morphic cohomology. The second is a semi-topological Conner-Floyd isomorphism relating semi-topological $K$-theory and semi-topological cobordism.

Voevodsky's slice spectral sequence is recalled in Section \ref{sub:coeff}. Write $\L^{*}$ for the Lazard ring, where the grading is such that $MU^{2q} = \L^{q}$.
An unpublished result of Hopkins-Morel implies that the slices  for $\MGL$  are $s_q\MGL =\Sigma_{\P^1}^{q}\MZ\otimes \L^{-q}$. A proof of this result has recently appeared in work of  Hoyois \cite{Hoyois}.  We have an exact couple given by setting
 $D_{2}^{p,q,b} = \ppsh_{-p-q,b}(f_{-q}\MGL)(X)$ and $E_{2}^{p,q,b} = \ppsh_{-p-q,b}(s_{-q}\MGL)(X) = H^{p-q}_{\mcal{M}}(X,\Z(b-q))\otimes \L^{q}$. The resulting spectral sequence
\begin{equation*}
E_{2}^{p,q,b}(alg) = H_{\mcal{M}}^{p-q}(X,\Z(b-q))\otimes \L^{q}\Longrightarrow
    \MGL^{p+q,b}(X)
\end{equation*}
is strongly convergent (see \cite[Lemmas 7.9,7.10]{Hoyois}).
Since $\MGL$ is a commutative ring spectrum  this is a multiplicative spectral sequence. 

We have cofiber sequences
$\Qsst (f_{q+1} \MGL) \to \Qsst (f_{q}\MGL) \to \Qsst (s_{q}\MGL)$, which  give rise to the
spectral sequence

\begin{equation*}\label{eqn:sstspseq}
E_{2}^{p,q,b}(sst) = L^{b-q}H^{p-q}(X)\otimes \L^{q}\Longrightarrow
    \MGL^{p+q,b}_{sst}(X),
\end{equation*}
which is strongly convergent, one may argue as in the proof of Theorem \ref{thm:toppt}. Alternatively using that $(\ppsh_{p,b}f_{q}\MGL)_{Nis} =0$ for $p-b<q$, 
see \cite[Lemma 7.10]{Hoyois}, one can show that $\ppsh_{p,b}\Qsst(f_{q}\MGL)(X) = 0$ for $p<q-b-\dim X$. 
It is multiplicative because $\Qsst$ is lax symmetric monoidal and so we have pairing of slice towers.

\begin{proposition}
The spectral sequence 
\begin{equation*}
   E_{2}^{p,q}(top) 
 = H^{p-q}(X(\C), \L^{q})
\Longrightarrow MU^{p+q}(X(\C)),
 \end{equation*}
associated to the cofiber sequences
\begin{equation*}
\L\Re_{\C}(f_{q+1}\MGL) \to \L\Re_{\C}(f_{q}\MGL) \to \L\Re_{\C}(s_{q}\MGL),
\end{equation*}
is the spectral arising from the Postnikov tower.
\end{proposition}
\begin{proof}
By \cite[Lemma 7.10]{Hoyois}, $f_{q}\MGL$ is ``topologically $q$-connected'' (in the terminology of \cite{Levine:comp}). Applying \cite[Theorem 5.2]{Levine:comp} to $f_{q}\MGL$, we conclude that $\L\Re_{\C}(f_{q}\MGL)$ is $2q-1$-connected. We have $\L\Re_{\C}(s_{q}\MGL) = \Sigma^{2q}H\pi_{2q}MU$ and so an easy inductive argument implies that $\L\RRe_{\C}$ applied to the slice tower for $\MGL$ agrees with the Postnikov tower.
\end{proof}

Immediate from the constructions is that the natural comparison maps
$$
\{E^{p,q,b}_r(alg)\}\to \{E^{p,q,b}_r(sst)\} \to \{E^{p,q}_r(top)\}
$$
are compatible with the multiplicative structures.

\begin{remark}\label{rem:fullpt}
 When $X=\C$, this comparison map of spectral sequence gives a complete calculation of the coefficients of $\MGL_{\sst}^{*,*}(\C)$, improving in this case the result of Theorem \ref{thm:toppt}.
Since $E^{p,q,b}_2(sst) = 0 = E^{p,q}_r(top)$ if $p\neq q$, both spectral sequences collapse. We have then that
$\MGL_{\sst}^{s,b}(\C) \to MU^{s}$ is an isomorphism if $s\leq 2b$ and 
 $\MGL_{\sst}^{s,b}(\C)=0$ if $s>2b$.
\end{remark}

\begin{theorem}
 Let $X$ be a smooth quasi-projective complex variety. If $b\geq \dim(X)$, the map  $\MGL^{s,b}_{\sst}(X) \to MU^{s}(X(\C))$ is an isomorphism.
\end{theorem}
\begin{proof}
The comparison map $ L^{m}H^{n}(X) \to H^{n}(X(\C))$ is an isomorphism for all $n$ when $m\geq \dim(X)$. Consider the comparison map $\{E^{p,q,b}_r(sst)\}\to  \{E^{p,q}_r(top)\}$. 
Since $t\geq \dim X$, if $b-q<\dim X$ then $q>0$ and both sides are zero. Otherwise $q\leq 0$ and so $ L^{b-q}H^{p-q}(X)\otimes \L^{q} \to H^{p-q}(X(\C), \L^{q})$
is an isomorphism for all $p-q$. 
We conclude that $E^{p,q,b}_2(sst)\to  E^{p,q}_2(top)$ is an isomorphism for all $p,q$  which implies the result.
\end{proof}

\begin{corollary}
 There is an element $s\in \MGL_{sst}^{0,1}(\C)$ such that we have an isomorphism 
of graded rings 
$$
(\oplus_{q}\MGL^{2q+*,q}(X))[s^{-1}] = \oplus_{q}MU^{2q+*}(X).
$$
\end{corollary}
\begin{proof}
In \cite[Proposition 5.6]{FW:ratisos} Friedlander-Mazur's $s$-operation in morphic cohomology is reinterpreted as the cup product with  an element $s\in L^1H^0(\C)$. 
Consider the commutative square
$$
\xymatrix{
\MGL_{\sst}^{0,1}(\C) \ar[r]\ar[d]& L^1H^0(\C)\ar[d]\\
 MU^0(\pt) \ar[r]& H^{0}(\pt),
}
$$
in which all arrows are isomorphisms. The element $s\in L^{1}H^{0}(\C)$ lifts to an element $s\in \MGL_{\sst}^{0,1}(\C)$. Since $s\in L^{1}H^{0}(\C)$ maps to $1\in H^{0}(\pt)$, and $1\in MU^{0}(\pt)$ maps to $1\in H^{0}(\pt)$, we see that the same is true for $s\in \MGL_{\sst}^{0,1}(\C)$. We have a map of rings as in the statement of the corollary and the isomorphism follows from the fact that  $MGL^{p,q}_{\sst}(X) \to MU^{p}(X(\C))$ is an isomorphism whenever $q\geq \dim(X)$.
\end{proof}

We now turn to a semi-topological Conner-Floyd isomorphism.  Write $\MGL^{i}_{\sst}(X) = \oplus_{p-2q=i}\MGL^{p,q}_{\sst}(X)$. 
\begin{theorem}
 For any smooth $X$, there is a natural isomorphism 
$$
\MGL_{\sst}^{*}(X)\otimes_{\MGL_{\sst}^{0}(\C)}K^{\sst}_{0}(\C)\iso K^{\sst}_{-*}(X).
$$
\end{theorem}
\begin{proof}
We show that for any compact motivic space $W$, there is a natural isomorphism
$P\MGL_{\sst}^{0,0}(W)\otimes_{P\MGL_{\sst}^{0,0}(\C)}\KGL_{\sst}^{0,0}(\C)\iso \KGL_{\sst}^{0,0}(W)$, where $P\MGL^{\sst}$ is the periodic semi-topological cobordism spectrum, see Section \ref{sub:cob}. 
The argument is similar to that of \cite[Theorem 5.3]{GepnerSnaith}, so we give the main points and refer to loc.~cit.~ for full details. 

Recall that $P\MGL^{\sst}\wkeq (\Qsst(\Sigma^{\infty}\BGL_{+}))[\beta^{-1}]$ and 
$\KGL^{\sst}\wkeq (\Sigma^{\infty}\BGmp)[\beta^{-1}]$. The map $P\MGL^{\sst}\to \KGL^{\sst}$ is induced under these equivalences by the determinant $\BGL\to \BGm$ which is split by the inclusion $\BGm\to \BGL$. This easily implies that $P\MGL_{\sst}^{0,0}(W)\otimes_{P\MGL_{\sst}^{0,0}(\C)}\KGL_{\sst}^{0,0}(\C)\iso \KGL_{\sst}^{0,0}(W)$ is surjective.

Write $J^{0}(W)=\ker([W,P\MGL_{\sst}]\to [W,\KGL_{\sst}])$. It suffices to show that the map $J^{0}(W)\otimes_{P\MGL_{sst}^{0,0}(\C)}J^{0}(\C)\to J^{0}(W)$ is surjective.
Since $W$ is compact, an element of $x\in J^{0}(W)$ is represented by a map $f_{x}:W\to \Sigma^{-2n,-n}\Qsst\Sigma^{\infty}\BGL_+$ for some $n$. In turn since $\BGL = \colim_{p,q}\Grass_{p,q}$, this element is represented by a map $f:W\to \Qsst(\Sigma^{\infty}(\Grass_{p,q})_+)$. Using this and that $\BGL\to \BGm$ is split, a diagram chase as in \cite[Theorem 5.3]{GepnerSnaith} completes the proof. 
\end{proof}

\section*{Acknowledgements}
We thank David Gepner, Christian Haesemeyer, Jens Hornbostel, Mircea Voineagu, and  Paul Arne {\O}stv{\ae}r for useful and interesting conversations related to various aspects of this work.  Some of the material here is a streamlined and generalized version of parts of the author's thesis, written under the guidance  of Eric Friedlander and we are especially grateful to him for all of his encouragement and help.   
\bibliographystyle{amsalpha} 
\bibliography{sst} 

\providecommand{\bysame}{\leavevmode\hbox to3em{\hrulefill}\thinspace}
\providecommand{\MR}{\relax\ifhmode\unskip\space\fi MR }
\providecommand{\MRhref}[2]{%
  \href{http://www.ams.org/mathscinet-getitem?mr=#1}{#2}
}
\providecommand{\href}[2]{#2}
\begin{thebibliography}{{Lev}12b}

\bibitem[Bei12]{Beilinson}
A.~Beilinson, \emph{Remarks on {G}rothendieck's standard conjectures},
  Regulators, Contemp. Math., vol. 571, Amer. Math. Soc., Providence, RI, 2012,
  pp.~25--32. \MR{2953406}

\bibitem[BF78]{BF}
A.~K. Bousfield and E.~M. Friedlander, \emph{Homotopy theory of {$\Gamma
  $}-spaces, spectra, and bisimplicial sets}, Geometric applications of
  homotopy theory ({P}roc. {C}onf., {E}vanston, {I}ll., 1977), {II}, Lecture
  Notes in Math., vol. 658, Springer, Berlin, 1978, pp.~80--130. \MR{513569
  (80e:55021)}

\bibitem[BK72]{BK}
A.~K. Bousfield and D.~M. Kan, \emph{Homotopy limits, completions and
  localizations}, Lecture Notes in Mathematics, Vol. 304, Springer-Verlag,
  Berlin, 1972. \MR{0365573 (51 \#1825)}

\bibitem[CS02]{CS:diagrams}
W.~Chach{\'o}lski and J.~Scherer, \emph{Homotopy theory of diagrams}, Mem.
  Amer. Math. Soc. \textbf{155} (2002), no.~736, x+90. \MR{1879153
  (2002k:55026)}

\bibitem[DI04]{DI:hyp}
D.~Dugger and D.~C. Isaksen, \emph{Topological hypercovers and
  {$\mathbb{A}^1$}-realizations}, Math. Z. \textbf{246} (2004), no.~4,
  667--689. \MR{MR2045835 (2005d:55026)}

\bibitem[Dug01]{dugger:simp}
D.~Dugger, \emph{Replacing model categories with simplicial ones}, Trans. Amer.
  Math. Soc. \textbf{353} (2001), no.~12, 5003--5027 (electronic). \MR{1852091
  (2002f:55043)}

\bibitem[FHW04]{FHW:sst}
E.~M. Friedlander, C.~Haesemeyer, and M.~E. Walker, \emph{Techniques,
  computations, and conjectures for semi-topological {$K$}-theory}, Math. Ann.
  \textbf{330} (2004), no.~4, 759--807. \MR{MR2102312}

\bibitem[FL92]{FL:algco}
E.~M. Friedlander and H.~B. Lawson, Jr., \emph{A theory of algebraic cocycles},
  Ann. of Math. (2) \textbf{136} (1992), no.~2, 361--428. \MR{MR1185123
  (93g:14013)}

\bibitem[FW01]{FW:compK}
E.~M. Friedlander and M.~E. Walker, \emph{Comparing {$K$}-theories for complex
  varieties}, Amer. J. Math. \textbf{123} (2001), no.~5, 779--810.
  \MR{MR1854111 (2002i:19004)}

\bibitem[FW02]{FW:sstfct}
\bysame, \emph{Semi-topological {$K$}-theory using function complexes},
  Topology \textbf{41} (2002), no.~3, 591--644. \MR{MR1910042 (2003g:19005)}

\bibitem[FW03]{FW:ratisos}
\bysame, \emph{Rational isomorphisms between {$K$}-theories and cohomology
  theories}, Invent. Math. \textbf{154} (2003), no.~1, 1--61. \MR{MR2004456
  (2004j:19002)}

\bibitem[GS09]{GepnerSnaith}
D.~Gepner and V.~Snaith, \emph{On the motivic spectra representing algebraic
  cobordism and algebraic {$K$}-theory}, Doc. Math. \textbf{14} (2009),
  359--396. \MR{2540697 (2011b:55004)}

\bibitem[Hel06]{thesis}
J.~Heller, \emph{Semi-topological cobordism for complex varieties}, ProQuest
  LLC, Ann Arbor, MI, 2006, Thesis (Ph.D.)--Northwestern University.
  \MR{2708564}

\bibitem[Hir03a]{Hir:loc}
P.~S. Hirschhorn, \emph{Model categories and their localizations}, Mathematical
  Surveys and Monographs, vol.~99, American Mathematical Society, Providence,
  RI, 2003. \MR{MR1944041 (2003j:18018)}

\bibitem[Hir03b]{H:loc}
\bysame, \emph{Model categories and their localizations}, Mathematical Surveys
  and Monographs, vol.~99, American Mathematical Society, Providence, RI, 2003.
  \MR{1944041 (2003j:18018)}

\bibitem[Hov01]{Hovey:Spt}
M.~Hovey, \emph{Spectra and symmetric spectra in general model categories}, J.
  Pure Appl. Algebra \textbf{165} (2001), no.~1, 63--127. \MR{1860878
  (2002j:55006)}

\bibitem[{Hoy}12]{Hoyois}
M.~{Hoyois}, \emph{{From algebraic cobordism to motivic cohomology}}, ArXiv
  e-prints (2012).

\bibitem[Isa05]{I:flasque}
D.~C. Isaksen, \emph{Flasque model structures for simplicial presheaves},
  $K$-Theory \textbf{36} (2005), no.~3-4, 371--395 (2006). \MR{2275013
  (2007j:18014)}

\bibitem[Jar97]{Jar:genet}
J.~F. Jardine, \emph{Generalized \'etale cohomology theories}, Progress in
  Mathematics, vol. 146, Birkh\"auser Verlag, Basel, 1997. \MR{MR1437604
  (98c:55013)}

\bibitem[Jar00]{Jar:motspt}
\bysame, \emph{Motivic symmetric spectra}, Doc. Math. \textbf{5} (2000),
  445--553 (electronic). \MR{MR1787949 (2002b:55014)}

\bibitem[KP13]{KrishnaPark}
A.~{Krishna} and J.~{Park}, \emph{{Semi-topologization in motivic homotopy
  theory and applications}}, ArXiv e-prints (2013).

\bibitem[Lev09]{Levine:cob}
M.~Levine, \emph{Comparison of cobordism theories}, J. Algebra \textbf{322}
  (2009), no.~9, 3291--3317. \MR{2567421 (2011e:14042)}

\bibitem[{Lev}12a]{Levine:comp}
M.~{Levine}, \emph{{A comparison of motivic and classical homotopy theories}},
  ArXiv e-prints (2012).

\bibitem[{Lev}12b]{Levine:conv}
\bysame, \emph{{Convergence of Voevodsky's slice tower}}, ArXiv e-prints
  (2012).

\bibitem[Mum70]{Mumford}
D.~Mumford, \emph{Abelian varieties}, Tata Institute of Fundamental Research
  Studies in Mathematics, No. 5, Published for the Tata Institute of
  Fundamental Research, Bombay, 1970. \MR{0282985 (44 \#219)}

\bibitem[MVW06]{MVW}
C.~Mazza, V.~Voevodsky, and C.~Weibel, \emph{Lecture notes on motivic
  cohomology}, Clay Mathematics Monographs, vol.~2, American Mathematical
  Society, Providence, RI, 2006. \MR{2242284 (2007e:14035)}

\bibitem[Pel11]{Pelaez}
P.~Pelaez, \emph{Multiplicative properties of the slice filtration},
  Ast\'erisque (2011), no.~335, xvi+289. \MR{2807904 (2012d:14034)}

\bibitem[PPR09a]{PPR:MGL2KGL}
I.~Panin, K.~Pimenov, and O.~R{\"o}ndigs, \emph{On the relation of
  {V}oevodsky's algebraic cobordism to {Q}uillen's {$K$}-theory}, Invent. Math.
  \textbf{175} (2009), no.~2, 435--451. \MR{2470112 (2010b:14039)}

\bibitem[PPR09b]{PPR:KGL}
\bysame, \emph{On {V}oevodsky's algebraic {$K$}-theory spectrum}, Algebraic
  topology, Abel Symp., vol.~4, Springer, Berlin, 2009, pp.~279--330.
  \MR{2597741 (2011c:19005)}

\bibitem[R{\O}08]{RO:MZ}
O.~R{\"o}ndigs and P.~A. {\O}stv{\ae}r, \emph{Modules over motivic cohomology},
  Adv. Math. \textbf{219} (2008), no.~2, 689--727. \MR{2435654 (2009m:14026)}

\bibitem[RS{\O}10]{RSO:strict}
O.~R{\"o}ndigs, M.~Spitzweck, and P.~A. {\O}stv{\ae}r, \emph{Motivic strict
  ring models for {$K$}-theory}, Proc. Amer. Math. Soc. \textbf{138} (2010),
  no.~10, 3509--3520. \MR{2661551 (2011h:14024)}

\bibitem[Shi00]{Shipley:THH}
B.~Shipley, \emph{Symmetric spectra and topological {H}ochschild homology},
  $K$-Theory \textbf{19} (2000), no.~2, 155--183. \MR{1740756 (2001h:55010)}

\bibitem[S{\O}09]{SO:bott}
M.~Spitzweck and P.~A. {\O}stv{\ae}r, \emph{The {B}ott inverted infinite
  projective space is homotopy algebraic {$K$}-theory}, Bull. Lond. Math. Soc.
  \textbf{41} (2009), no.~2, 281--292. \MR{2496504 (2010f:19005)}

\bibitem[SV00]{SV:rel}
A.~Suslin and V.~Voevodsky, \emph{Relative cycles and {C}how sheaves}, Cycles,
  transfers, and motivic homology theories, Ann. of Math. Stud., vol. 143,
  Princeton Univ. Press, Princeton, NJ, 2000, pp.~10--86. \MR{1764199}

\bibitem[Voe98]{Voev:ICM}
V.~Voevodsky, \emph{{$\mathbf A\sp 1$}-homotopy theory}, Proceedings of the
  International Congress of Mathematicians, Vol. I (Berlin, 1998), no. Extra
  Vol. I, 1998, pp.~579--604 (electronic). \MR{MR1648048 (99j:14018)}

\bibitem[Voe04]{Voev:slice}
\bysame, \emph{On the zero slice of the sphere spectrum}, Tr. Mat. Inst.
  Steklova \textbf{246} (2004), no.~Algebr. Geom. Metody, Svyazi i Prilozh.,
  106--115. \MR{2101286 (2005k:14042)}

\bibitem[Voe10]{V:niscdh}
\bysame, \emph{Unstable motivic homotopy categories in {N}isnevich and
  cdh-topologies}, J. Pure Appl. Algebra \textbf{214} (2010), no.~8,
  1399--1406. \MR{2593671 (2011e:14041)}

\bibitem[Wei99]{Weibel:prod}
C.~Weibel, \emph{Products in higher {C}how groups and motivic cohomology},
  Algebraic {$K$}-theory ({S}eattle, {WA}, 1997), Proc. Sympos. Pure Math.,
  vol.~67, Amer. Math. Soc., Providence, RI, 1999, pp.~305--315. \MR{1743246
  (2001a:14017)}

\bibitem[Yag11]{Yagunov:rigid}
S.~Yagunov, \emph{Remark on rigidity over several fields}, Homology Homotopy
  Appl. \textbf{13} (2011), no.~2, 159--164. \MR{2854332}

\end{thebibliography}
\end{document}